\documentclass[10pt]{article}
\RequirePackage[left=28mm,right=28mm,top=35mm,bottom=30mm]{geometry}
\usepackage[utf8]{inputenc}
\usepackage[english]{babel}
\usepackage[]{amsmath}
\usepackage{amssymb}
\usepackage{amsthm}
\usepackage[shortlabels]{enumitem} 
\usepackage[square,sort,comma,numbers]{natbib}
\usepackage{afterpage}
\usepackage{graphicx} 
\usepackage{tikz-cd} 
\usepackage{stmaryrd} 
\usepackage{url}
\usepackage{hyperref}
\usepackage{multicol}
\usepackage{xcolor}
\usepackage[toc, title]{appendix}

\DeclareFontFamily{U}{mathx}{\hyphenchar\font45}
\DeclareFontShape{U}{mathx}{m}{n}{
	<5> <6> <7> <8> <9> <10>
	<10.95> <12> <14.4> <17.28> <20.74> <24.88>
	mathx10
}{}
\DeclareSymbolFont{mathx}{U}{mathx}{m}{n}
\DeclareFontSubstitution{U}{mathx}{m}{n}
\DeclareMathAccent{\widecheck}{0}{mathx}{"71}
\DeclareMathAccent{\wideparen}{0}{mathx}{"75}

\tikzcdset{scale cd/.style={every label/.append style={scale=#1},
		cells={nodes={scale=#1}}}}
	
	\usepackage{mathtools}

\newcommand{\End}{\operatorname{End}}
\newcommand{\Hom}{\operatorname{Hom}}
\newcommand{\gldim}{\operatorname{gldim}}

\newcommand{\Ext}{\operatorname{Ext}}
\newcommand{\Tor}{\operatorname{Tor}}
\newcommand{\add}{\!\operatorname{add}}

\newcommand{\m}{\!\operatorname{-mod}} 
\newcommand{\M}{\!\operatorname{-Mod}} 
\newcommand{\proj}{\!\operatorname{-proj}}
\newcommand{\free}{\!\operatorname{-free}}
\newcommand{\St}{\Delta}
\newcommand{\Cs}{\nabla}
\renewcommand{\L}{\Lambda}
\renewcommand{\l}{\lambda}
 
\newcommand{\id}{\operatorname{id}}
\newcommand{\mi}{\mathfrak{m}}

\newcommand{\MaxSpec}{\operatorname{MaxSpec}} 
\newcommand{\Stsim}{\tilde{\St}}
\newcommand{\im}{\!\operatorname{im}} 

\newcommand{\coker}{\operatorname{coker}}
\newcommand{\sumSt}{\underset{\l\in\L}{\bigoplus}}

\newcommand{\Cssim}{\tilde{\Cs}}

\newtheorem{numberingthm}{Theorem}[section] 
\theoremstyle{definition}
\newtheorem{Def}[numberingthm]{Definition}

\theoremstyle{plain}
\newtheorem{Prop}[numberingthm]{Proposition}
\newtheorem{Theorem}[numberingthm]{Theorem}
\newtheorem{Cor}[numberingthm]{Corollary}
\newtheorem{Lemma}[numberingthm]{Lemma}
\newtheorem{Remark}[numberingthm]{Remark}

\theoremstyle{remark}

\theoremstyle{empty}
\newtheorem*{thmintroduction}{Theorem}
\newtheorem*{propintroduction}{Proposition}
\newtheorem*{corintroduction}{Corollary}
\providecommand{\keywords}[1]
{\scriptsize
	\textbf{\textit{Keywords:}} #1 \normalsize \hfill
}
\providecommand{\msc}[1]
{\scriptsize
	\textbf{\textit{2020 Mathematics Subject Classification:}} #1 \normalsize \hfill
}

\title{Characteristic tilting modules and Ringel duality \\in the Noetherian world}
\author{Tiago Cruz} 
\date{}

\newcommand{\Address}{{
		\bigskip
		\footnotesize
		
		TIAGO CRUZ,\par \textsc{Institute of Algebra and Number Theory}\par \textsc{University of Stuttgart,}\par \textsc{Pfaffenwaldring 57, 70569 Stuttgart, Germany,}\par\nopagebreak
		\textit{E-mail address}, T.~Cruz: \texttt{tiago.cruz@mathematik.uni-stuttgart.de}
		}}
		
		\allowdisplaybreaks

\begin{document}

\maketitle

\begin{abstract}
	The foundations of Ringel duality for split quasi-hereditary algebras over commutative Noetherian rings are strengthened. Several descriptions and properties of the smallest resolving subcategory containing all standard modules over split quasi-hereditary algebras over  commutative Noetherian rings are provided.
	In particular, given two split quasi-hereditary algebras $A$ and $B$, we prove that any exact equivalence between the smallest resolving subcategory containing all standard modules over $A$ and the smallest resolving subcategory containing all standard modules over $B$ lifts to a Morita equivalence between $A$ and $B$ which preserves the quasi-hereditary structure.
\end{abstract}


	\section{Introduction}\label{Existence and properties of costandard modules}
	
	(Split) quasi-hereditary algebras over commutative Noetherian rings were introduced in \cite{CLINE1990126} and a module theoretical approach to them was initiated in \cite{Rouquier2008}. Paradigmatic examples of quasi-hereditary algebras are Schur algebras and blocks of the BGG category $\mathcal{O}$ of a complex semi-simple Lie algebra, and both of these examples admit integral versions (see for example \cite{p2paper}). Furthermore, going integrally is important for modular representation theory and there are phenomena in quasi-hereditary algebras over fields that become easier to visualise and comprehend by studying split quasi-hereditary algebras over commutative Noetherian rings (see for example \citep[8.2]{Cr2}).

	Continuing \cite{cruz2021cellular}, this paper is about consolidating the foundations of the theory of split highest weight categories. The focus is on the structure of $\mathcal{F}(\Stsim)$, in particular on understanding why $\mathcal{F}(\Stsim)$ is a resolving subcategory of $A\m$  in the integral setup and how a choice of a characteristic tilting module determines $\mathcal{F}(\Stsim)$ also in the integral setup. 
	Some proofs carry over almost unchanged from the classical case to the integral setup while others require more work. Indeed, in the general setting, there are obstructions that make it impossible to mimic the proofs of the classical case to our situation. For instance, the categories $\mathcal{F}(\Stsim)$ and $A\m$ are not, in general, Krull-Schmidt categories and not all objects admit a projective cover.
	Furthemore some proofs will follow ideas of \cite{Rouquier2008}, with some modifications that are hoped to make the contents more accessible. 
	 This paper provides background material needed in \cite{p2paper, Cr2}.
	
	The main objects in the theory of quasi-hereditary algebras are the standard modules, and these are precisely the simple objects in the exact category of all modules admitting a filtration by standard modules. In the integral setup, this category does not necessarily contain all projective modules nor are standard modules necessarily indecomposable. To fix this issue  Rouquier in \cite{Rouquier2008} studied the larger category $\mathcal{F}(\Stsim)$ in a split quasi-hereditary algebra. This one is defined as the category whose modules admit a finite filtration by modules in the additive closure of standard modules. 
	
	Our main result (see Theorem \ref{standardsgivesthewholealgebra}) establishes that the subcategory $\mathcal{F}(\Stsim)$ completely determines the split quasi-hereditary structure on $A$. 
	
	\begin{thmintroduction}[\ref{standardsgivesthewholealgebra}]
		Let $A$ and $A'$ be two split quasi-hereditary algebras with standard modules $\{\Delta(\lambda)\colon \lambda\in \Lambda\}$ and $\{\Delta'(\gamma)\colon \gamma\in \Gamma\}$, respectively. Then, any exact equivalence $\Phi\colon \mathcal{F}(\Stsim)\rightarrow \mathcal{F}(\Stsim')$ lifts to a Morita equivalence $F\colon A\m\rightarrow A'\m$ such that the following diagram of functors is commutative, up to isomorphism,
	\begin{equation*}
		\begin{tikzcd}
			A\m \arrow[r, "F"] & A'\m \\
			\mathcal{F}(\Stsim) \arrow[u, hookrightarrow] \arrow[r, "\Phi"] & \mathcal{F}(\Stsim') \arrow[u, hookrightarrow]
		\end{tikzcd}.
	\end{equation*}	Moreover, $F$ is an equivalence of highest weight categories in the sense of \cite{Rouquier2008}.
	\end{thmintroduction}

	
	Split quasi-hereditary algebras behave quite well under change of rings techniques (see for example \citep[Subsection 3.1]{cruz2021cellular}) and they form a class closed under taking the opposite algebra. To further illustrate that point in Proposition \ref{tensorproductstandcostand} we prove the following.
	\begin{propintroduction}[\ref{tensorproductstandcostand}]
	Let $R$ be a commutative Noetherian ring. Let $A$ be a split quasi-hereditary $R$-algebra with standard modules $\{\Delta(\lambda)\colon \lambda\in \Lambda\}$. Suppose that $N$ admits a finite filtration by standard modules over the opposite algebra $A^{op}$.
	Then, 
		 the functor $N\otimes_A -\colon \mathcal{F}(\Stsim)\rightarrow R\proj$ is a well-defined exact functor.
	\end{propintroduction}
This fact explains the existence of many base change properties used in the literature of split quasi-hereditary algebras, for instance \citep[2.5]{Koenig2001}.

	In the integral setup, basic characteristic tilting modules are not necessarily unique. However, adapting the classical theory to the integral setup, we establish that Ringel duality exists in the integral setup. That is, the Ringel dual of the Ringel dual is again, up to isomorphism, the original split quasi-hereditary algebra. Further, using Theorem \ref{standardsgivesthewholealgebra} we prove that the Ringel dual admits the following characterisation.

	\begin{corintroduction}[\ref{standardstocostandardsringel}]Let $A$ be a split quasi-hereditary algebra with standard modules $\{\Delta(\lambda)\colon \lambda\in \Lambda\}$ and let $B$ be a split quasi-hereditary algebra with costandard modules $\{\mho(\chi)\colon \chi\in X\}$.
	Then, 
		$B$ is a Ringel dual of $A$ if and only if there exists an exact equivalence between the categories $\mathcal{F}(\Stsim)$ and $\mathcal{F}(\tilde{\mho}).$
	\end{corintroduction}

	The structure of this paper is as follows: 
	We will start in Section \ref{Existence of costandard modules} by looking at split quasi-hereditary algebras from the point of view of costandard modules and the exact category of modules having a finite filtration with factors in the additive closure of costandard modules, $\mathcal{F}(\Cssim)$. In particular, we present in Theorem \ref{quasihereditaryintermsofcostandards} an alternative characterization of split quasi-hereditary algebras using costandard modules. In Section \ref{On Ext-projective and Ext-injective objects of}, we see that $\mathcal{F}(\Stsim)$ and $\mathcal{F}(\Cssim)$ completely determine each other by forming an Ext-orthogonal pair in $A\m\cap R\proj$. In Subsection \ref{Tilting modules}, a characteristic tilting module is constructed so that it belongs both in $\mathcal{F}(\Stsim)$ and in $\mathcal{F}(\Cssim)$ while at the same time every $M\in \mathcal{F}(\Stsim)$ admits a coresolution by modules in the additive closure of a module $T\in \mathcal{F}(\Stsim)\cap \mathcal{F}(\Cssim)$. We finish Subsection \ref{Tilting modules} by showing that although no uniqueness of basic characteristic tilting modules can be expected, distinct characteristic tilting modules have the same additive closure. In Section \ref{Change of rings and filtrations of}, we study the costandard modules and characteristic tilting modules under change of ground ring. In particular, endomorphism algebras of modules in the additive closure of a characteristic tilting module have a base change property. 
	In Section \ref{Structure of split quasi-hereditary algebras}, we establish that exact equivalences between exact subcategories of  modules having a finite filtration by summands of direct sum of copies of standard modules over split highest weight categories can be lifted to exact equivalences between the ambient split highest weight categories. In Section \ref{Ringel duality}, we establish that split quasi-hereditary algebras also occur in pairs in the integral setup. Further, we prove that checking Ringel self-duality for split quasi-hereditary algebras over local commutative Noetherian rings can be reduced to checking Ringel self-duality in the finite-dimensional case. We conclude Section \ref{Ringel duality} by providing a brief theoretical argument that many integral cellular structures arise as endomorphism algebras of direct summands of characteristic tilting modules. 
	
	\section{Preliminaries}
	
	In the following, we are going to present the terminology and notation to be used throughout the paper. In particular, we will follow the same terminology and notation presented in \cite{cruz2021cellular, p2paper}, and hence for undefined terminology or notation we refer the reader to \cite{cruz2021cellular, p2paper}. Let $R$ be a commutative Noetherian ring with identity. We will denote by $\MaxSpec R$ the set of all maximal ideals of $R$. We say that an $R$-algebra $A$ is a \textbf{projective Noetherian $R$-algebra} if the regular module $A$ is finitely generated and projective as $R$-module. We will denote by $A\m$ the category of finitely generated left $A$-modules and by $A\proj$ the full subcategory of $A\m$ whose modules are projective. Given $M\in A\m$, we denote by $\add_A M$ the additive closure of $M$. By \textbf{generator} (over $A$), we mean a module $M\in A\m$ whose additive closure contains the regular module $A$. We say that a module is a \textbf{progenerator} if it is both a projective module and a generator. Given an $R$-algebra, we will denote by $A^{op}$ the opposite algebra of $A$ and by $D$ the standard duality functor $D=\Hom_R(-, R)\colon A\m\rightarrow A^{op}\m$. We say that an exact sequence of $A$-modules is \textbf{$(A, R)$-exact} if it splits as exact sequence of $R$-modules. An \textbf{$(A, R)$-monomorphism} is a monomorphism of $A$-modules which splits as monomorphism of $R$-modules.   We say that a module $M$ in $A\m\cap R\proj$ is \textbf{$(A, R)$-injective} if $M\in \add DA$. 
	Given $\mi\in \MaxSpec R$ and $M\in A\m$, we will denote by $M_\mi$ the localisation of $M$ at the maximal ideal $\mi$. Further, we will write $M(\mi)$ to denote the module $M/\mi M\simeq R(\mi)\otimes_R M$, where $R(\mi)$ denotes the field $R/\mi$. Similarly, given an $R$-algebra $A$, we can consider the algebras $A_\mi$ and $A(\mi)$.
		
	A module $M$ over $R$ is called \textbf{invertible} if there exists an $R$-module $N$ such that $M\otimes_R N\simeq R$. The collection of all invertible $R$-modules forms a group known as the \textbf{Picard group} of $R$ which we will denote it by $Pic(R)$.

	\begin{Remark}\label{invertiblemodulesresidue}
		Let $M\in R\proj$. $M$ is an invertible $R$-module if and only if $M(\mi)\simeq R(\mi)$ for every $\mi\in \MaxSpec R$. In fact, for all $\mi\in \MaxSpec R$ there exists $n_\mi\in \mathbb{N}\cup \{ 0\}$ so that $M_\mi\simeq R_\mi^{n_\mi}$ and so $R(\mi)\simeq M_\mi(\mi)\simeq R(\mi)^{n_\mi}$ for all $\mi\in \MaxSpec R$. 
	\end{Remark}

	Recall the following fact that covariant Hom functors preserve finite direct sums, and hence commutes with  tensoring with a projective module over the ground ring.
	
	\begin{Lemma}\label{tensorprojcommutingonHom}
Let $A$	be a projective Noetherian $R$-algebra.	Let $M, N\in A\m$ and $U\in R\proj$. Then, the $R$-homomorphism $$\varsigma_{M, N, U}\colon \Hom_A(M, N)\otimes_R U\rightarrow \Hom_A(M, N\otimes_R U),$$ given by ${g\otimes u\mapsto g(-)\otimes u}$ is an $R$-isomorphism.
	\end{Lemma}
	\begin{proof}
		Since for all modules $U_1, U_2\in R\m$ there are commutative diagrams
		\begin{equation}
			\begin{tikzcd}
				\Hom_A(M, N)\otimes_R U_1\oplus \Hom_A(M, N)\otimes_RU_2\arrow[r, "\varsigma_{M, N, U_1}\oplus \varsigma_{M, N, U_2}", outer sep=0.75ex, swap] \arrow[d, "\simeq"] & \Hom_A(M, N\otimes_R U_1)\oplus \Hom_A(M, N\otimes_R U_2) \arrow[d, "\simeq"] \\
				\Hom_A(M, N)\otimes_R (U_1\oplus U_2)\arrow[r, "\varsigma_{M, N, U_1\oplus U_2}"] & \Hom_A(M, N\otimes_R (U_1\oplus U_2))
			\end{tikzcd},
		\end{equation}
		it is enough to show that $\varsigma_{M, N, R}$ is an $R$-isomorphism. But, this isomorphism is obtained by regarding $\varsigma_{M, N, R}$ in the following commutative diagram
		\begin{equation}
			\begin{tikzcd}
				\Hom_A(M, N)\otimes_R R\arrow[r, "\varsigma_{M, N, R}"] \arrow[d, "\mu_{\Hom_A(M{,}N)}"]& \Hom_A(M, N\otimes_R R)\arrow[d, "\Hom_A(M {,}\mu_N)"] \\
				\Hom_A(M, N)\arrow[r, equal] & \Hom_A(M, N)
			\end{tikzcd},
		\end{equation}where $\mu_X$ denotes the multiplication map for any $R$-module $X$.
		In fact, for all $f\in \Hom_A(M, N)$,
		\begin{align}
			\Hom_A(M, \mu_N)\circ \varsigma_{M, N, R}(f\otimes 1_R)(m)=\mu_N\circ \varsigma_{M, N, R}(f\otimes 1_R)(m)=\mu_N(f(m)\otimes_R 1_R)=f(m), \ m\in M.\nonumber
		\end{align}
		Hence, $\varsigma_{M, N, R}$ is an $R$-isomorphism.
	\end{proof}

Split quasi-hereditary algebras over Noetherian commutative rings were introduced in \cite{CLINE1990126}. We recall now the module-theoretical definition of split quasi-hereditary algebras given in \cite{Rouquier2008}.

\begin{Def}\label{qhdef}
	Given a projective Noetherian $R$-algebra $A$ and a collection of finitely generated left $A$-modules $\{\St(\l)\colon \l\in \L\}$ indexed by a  poset $\L$, we say that $(A, \{\Delta(\lambda)_{\lambda\in \Lambda}\})$ is a \textbf{split quasi-hereditary $R$-algebra} if the following conditions hold:
	\begin{enumerate}[(i)]
		\item The modules $\St(\l)$ belong to $A\m\cap R\proj$.
		\item Given $\l, \mu\in \L$, if $\Hom_A(\St(\l), \St(\mu))\neq 0$, then $\l\leq\mu$.
		\item $\End_A(\St(\l))\simeq R$, for all $\l\in\L$.
		\item Given $\l\in\L$, there is $P(\l)\in A\proj$ and an exact sequence $0\rightarrow C(\l)\rightarrow P(\l)\rightarrow \St(\l)\rightarrow 0$ such that $C(\l)$ has a finite filtration by modules of the form $\St(\mu)\otimes_R U_\mu$ with $U_\mu\in R\proj$ and $\mu>\l$. 
		\item $\add_A \left( \sumSt P(\l)\right)= A\proj$.
	\end{enumerate}
\end{Def} It is common also to say that under these conditions $(A\m, \{\Delta(\lambda)_{\lambda\in \Lambda}\})$ is a \textbf{split highest weight category.}

 Let $(A, \{\Delta(\lambda)_{\lambda\in \Lambda}\})$ be a split quasi-hereditary algebra. 
We will write $\St$ to denote the set ${\{\St(\l)\colon \l\in \L \}}$ and $\Stsim$ to denote the set ${\{\St(\l)\otimes_R U_\l\colon \l\in \L, U_\l\in R\proj \}}$. Given a subposet $\L'\subset \L$ we will denote by $\Stsim_{\L\setminus \L'}$ the set $\{\St(\l)\otimes_R U_\l\colon \l\in \L\setminus \L', \ U_\l\in R\proj \}$. By $\Stsim_{\mu>\l}$ we will denote the set \linebreak${\{\St(\mu)\otimes_R U_\mu\colon \mu\in \L, \ \mu>\l, U_\mu\in R\proj \}}$. Similarly, we will use the notations $\Stsim_{\mu\neq \l}$, $\Stsim_{\mu<\l}$, and the respective variants replacing $\Stsim$ with $\St$. For each $\mi\in \MaxSpec R$, by $\St(\mi)$ we mean the set ${\{\St(\l)(\mi) \colon \l\in \L \}}$. Given a set of $A$-modules $\theta$, we will denote by $\mathcal{F}(\theta)$ the full subcategory of $A\m$ whose modules admit a filtration $0=M_{t+1}\subset M_{t}\subset \cdots \subset M_1=M$ with $M_i/M_{i+1}\in \theta$, $1\leq i\leq t$, for some natural number $t$.

	By $\mathcal{M}(A)$ we denote the set of isomorphism classes of projective $R$-split $A$-modules. Here by a \textbf{projective $R$-split $A$-module} we mean a projective $A$-module $L$ being a progenerator as $R$-module so that for each $M\in A\proj$ the canonical morphism 
\begin{align}
	\tau_{L, M}\colon L\otimes_R \Hom_A(L, M)\rightarrow M, \quad l\otimes f\mapsto f(l),
\end{align}is an $(A, R)$-monomorphism. 
Projective $R$-split $A$-modules are useful objects that allow many properties of split quasi-hereditary algebras to be proved using induction methods. Indeed, given a projective Noetherian $R$-algebra $A$ and a collection of finitely generated left $A$-modules $\{\St(\l)\colon \l\in \L\}$ indexed by a  poset $\L$ with maximal element $\alpha$, $(A, \{\Delta(\lambda)_{\lambda\in \Lambda}\})$  is a split quasi-hereditary algebra if and only if $\Delta(\alpha)\in \mathcal{M}(A)$ and $(A/J, \{\Delta(\lambda)_{\lambda\in \Lambda\setminus \{\alpha\}}\})$ is a split quasi-hereditary $R$-algebra with $J=\im \tau_{\St(\alpha)}$ (see \citep[Lemma 4.12]{Rouquier2008} or \citep[Lemma B.0.4]{cruz2021cellular}). The ideal $J$ in this setup is called \textbf{split heredity ideal} and all split heredity ideals arise in this way (see for example \citep[Proposition A.2.5]{cruz2021cellular}).

Split quasi-hereditary algebras can also be defined through the existence of split heredity chains.  A chain of ideals $0\subset J_t\subset J_{t-1}\subset \cdots \subset J_1\subset A$ is called a split heredity chain if each ideal $J_i/J_{i+1}$ is a split heredity ideal in $A/J_{i+1}$ for $1\leq i\leq t$. Then, a projective Noetherian $R$-algebra $A$ is split quasi-hereditary if and only if it admits a split heredity chain (see for example \citep[Theorem 4.6]{Rouquier2008} or \citep[Theorem 3.3.4]{cruz2021cellular}). The following remark is quite useful to relate tensor products  of modules over a split quasi-hereditary algebra and over a quotient of a split quasi-hereditary algebra.

\begin{Remark}Let $A$	be a projective Noetherian $R$-algebra and let $J$ be an ideal of $A$. \label{remarkfullsubcategorytensor}
	If $Y\in A/J\m$ and $X\in (A/J)^{op}\m$ so that $X\otimes_A Y, X\otimes_{A/J} Y\in R\proj$, then $X\otimes_A Y\simeq X\otimes_{A/J} Y.$ In fact, $D(X\otimes_A Y)\simeq \Hom_A(Y, DX)\simeq \Hom_{A/J}(Y, DX)\simeq D(X\otimes_{A/J} Y)$ (see also \citep[Proposition 2.1]{CRUZ2022410}).
\end{Remark}

	\section{Existence of costandard modules}\label{Existence of costandard modules}
	In this section, we present an equivalent definition of split quasi-hereditary algebras over commutative Noetherian rings $R$ using costandard modules and $(A, R)$-injective modules.
	\begin{Prop}\label{existenceofcostandardmodules}
		
		Let $(A, \{\Delta(\lambda)_{\lambda\in \Lambda}\})$ be a split quasi-hereditary algebra. Then, there is a set $\{\Cs(\l) \}_{\l\in\L}$ of $A$-modules, unique up to isomorphism, with the following properties:
		
		\begin{enumerate}[(i)]
			\item $(A^{op}, \{D\Cs(\lambda)_{\lambda\in \Lambda}\})$ is a split quasi-hereditary algebra.
			\item Given $\l, \beta\in \L$, then $\Ext_A^i(\St(\l), \Cs(\beta))=\begin{cases}
				R \ \text{ if } i=0 \text{ and } \lambda=\beta\\ 0 \text{ otherwise}
			\end{cases}.$
		\end{enumerate} 
	\end{Prop}
	\begin{proof}
		Let $(A, \{\Delta(\lambda)_{\lambda\in \Lambda}\})$ be a split quasi-hereditary algebra and let $\L\rightarrow \{1, \ldots, t\}$, $\l\mapsto i_\l$ be an increasing bijection. $A$ is split quasi-hereditary with some heredity chain $0=J_{t+1}\subset J_t\subset \cdots \subset J_1=A$ (see for example \citep[Theorem 3.3.4]{cruz2021cellular} or \citep[Theorem 4.16]{Rouquier2008}).  By \citep[Theorem 3.3.8]{cruz2021cellular}, $A^{op}$ is split quasi-hereditary with split heredity chain $0=J^{op}_{t+1}\subset J_t^{op}\subset \cdots \subset J_1^{op}=A^{op}$. Again by \citep[Theorem 4.16]{Rouquier2008}, $(A^{op}\m, \{\Delta^*(\lambda)_{\lambda\in \Lambda}\})$ is a split highest weight category. 
		
		Let $\l, \beta\in \L$.	Assume $\beta\not>\l$.  As $\St^*(\l)\in \mathcal{M}(A^{op}/J^{op}_{i_\l+1})$, we obtain that $\St^*(\beta)\in  A^{op}/J^{op}_{i_\l}\m$ because of $A^{op}/J^{op}_{i_\l+1}/J^{op}_{i_\l}/J^{op}_{i_\l+1}\m\simeq A^{op}/J^{op}_{i_\l}\m$. Thus, $D\St^*(\beta)\in A/J_{i_\l}\m$. Since $\St(\l)\in \mathcal{M}(A/J_{i_\l+1})$ we obtain that $\Ext_A^{i\geq0}(\St(\l), D\St^*(\beta))=0$ and $\Ext_A^{i>0}(\St(\l), D\St^*(\l))=0$.
		By symmetry,  $\Ext_{A^{op}}^{i\geq 0}(\St^*(\l), D\St(\beta))=0$.
		By \citep[Proposition 2.2, Lemma 2.15]{CRUZ2022410}, if $\l\neq \beta$, then
		\begin{align}
			\Ext_A^{i\geq0}(\St(\l), D\St^*(\beta))\simeq \Ext_{A^{op}}^{i\geq0}(\St^*(\beta), D\St(\l))=0.
		\end{align}
		
		The $R$-module $U_\l=\Hom_A(\St(\l), D\St^*(\l))$ is invertible over $R$ (see for example \citep[proof of Proposition 4.19]{Rouquier2008}).  
		For each $\l\in \L$, define	$\Cs(\l)=DU_\l\otimes_R D\St^*(\l)$. Observe that $\Cs(\l)\in \add_A D\St^*(\l)$. It follows that $\Ext_A^{i\geq 0}(\St(\l), \Cs(\beta))=\Ext_A^{i>0}(\St(\l), \Cs(\l))=0$ if $\l\neq \beta$. For (ii), it remains to show that $\Hom_A(\St(\l), \Cs(\l))\simeq R$ for every $\l\in \L$.

		Recall that $\End_R(\St^*(\l))\simeq \Hom_{A}(J_{i_\l}/J_{i_\l +1}, A/J_{i_\l +1})$ (see \citep[Remark A.1.5]{cruz2021cellular}).	By Tensor-Hom adjunction, $U_\l\simeq D(\St^*(\l)\otimes_A \St(\l))$ and recall Remark \ref{remarkfullsubcategorytensor}  with $\Cs(\l)\in R\proj$ and $\St(\l)\in A/J_{i_\l+1}\proj$.
		Thus combining all these facts with \citep[Lemma 2.1.3]{p2paper} we obtain that
		\begin{align*}
			\Hom_A(\St(\l), \Cs(\l))&\simeq \Hom_{A}(\St(\l), \St^*(\l)\otimes_A \St(\l)\otimes_R D\St^*(\l))\\&\simeq \Hom_A(\St(\l), \Hom_R(\St^*(\l), R)\otimes_R \St^*(\l)\otimes_A \St(\l))\\&\simeq  \Hom_A(\St(\l), \Hom_R(\St^*(\l), \St^*(\l)) \otimes_A \St(\l))\\
			&\simeq \Hom_A(\St(\l), \Hom_{A/J_{i_\l +1}}(J_{i_\l}/ J_{i_\l +1}, A/ J_{i_\l +1})\otimes_{A/ J_{i_\l +1}} \St(\l))\\
			&\simeq \Hom_A(\St(\l), \Hom_{A}(J_{i_\l}/ J_{i_\l +1}, \St(\l)))\simeq  \Hom_A(J_{i_\l}/J_{i_\l +1}\otimes_A \St(\l), \St(\l))\\&\simeq \Hom_A(\St(\l)\otimes_R \Hom_A(\St(\l), A/J_{i_\l +1})\otimes_{A/J_{i_\l +1}} \St(\l), \St(\l))\\&\simeq \Hom_A(\St(\l)\otimes_R \Hom_A(\St(\l), \St(\l)), \St(\l) )\simeq \End_A(\St(\l))\simeq R.
		\end{align*}
		
		By Tensor-Hom adjunction and $\St^*(\l)\in R\proj$, 
		\begin{align*}
			D\Cs(\l)\simeq D(DU_\l\otimes_R D\St^*(\l))\simeq \Hom_R(D\St^*(\l), U_\l ) \simeq DD\St^*(\l)\otimes_R U_\l\simeq \St^*(\l)\otimes_R U_\l.
		\end{align*}Hence, $D\Cs(\l)_\mi\simeq \St^*(\l)_\mi$ for every $\mi\in \MaxSpec R$. By Theorem 3.1.3 of \cite{cruz2021cellular}, (i) follows.
		
		For the uniqueness part, we refer to \citep[proof of Proposition 4.19]{Rouquier2008}). 
	\end{proof} 
	From now on, we will denote by $\St_{op}(\l)$, $\l\in \L$, the standard modules of the opposite algebra of $(A, \{\Delta(\lambda)_{\lambda\in \Lambda}\})$.
	
	\begin{Cor}\label{dualofstandardiscostandard}
		Let $(A, \{\Delta(\lambda)_{\lambda\in \Lambda}\})$ be a split quasi-hereditary algebra. Then, $\{D\St(\l)\colon \l\in\L \}$ are costandard modules in $A^{op}$.
	\end{Cor}
	\begin{proof}
		Note that $\left( (A^{op})^{op}\m, \{DD\St(\l)_{\l\in\L}\}\right)=(A, \{\St(\l)_{\l\in\L}\}) $ is a split highest weight category. By  \citep[Proposition 2.2, Lemma 2.15]{CRUZ2022410}, for any $\l, \beta\in \L$,  \begin{align}
			\Ext_{A^{op}}^i(D\Cs(\l), D\St(\beta))\simeq \Ext_A^i(\St(\beta), \Cs(\l))=\begin{cases}
				R \text{ if } \l=\beta, \ i=0\\
				0 \text{ otherwise}.
			\end{cases} 
		\end{align}
		By the uniqueness of costandard modules in Proposition \ref{existenceofcostandardmodules}, the result follows.
	\end{proof}
	
	\begin{Remark}\label{nakayamaonstandardmaximal}
		If $\l\in \L$ is maximal, then $D\Hom_A(\St(\l), A)\simeq \Cs(\l)$. In fact, $\Hom_A(\St(\l), A)\in \mathcal{M}(A^{op})$ by Lemma A.1.8 of \cite{cruz2021cellular} and \begin{align}
			\Hom_A(\St(\l), D\Hom_A(\St(\l), A))\simeq \Hom_{A^{op}}(\Hom_A(\St(\l), A), D\St(\l))\simeq D\St(\l)\otimes_A \St(\l)\simeq R.
			\tag*{\qedhere}	\end{align}
	\end{Remark}
	
	\begin{Prop}\label{relativeinjectiveshavefiltrations}
		Let $(A, \{\Delta(\lambda)_{\lambda\in \Lambda}\})$ be a split quasi-hereditary algebra. Let $M\in A\m$ such that $M$ is $(A, R)$-injective and projective over $R$. Let $\L\rightarrow \{1, \ldots, t\}$, $\l\mapsto i_\l$ be an increasing bijection and set $\St_{i_\l}:=\St(\l)$. Then, there exists a filtration $
		0\subset I_1\subset \cdots \subset I_n=M,$ with $ I_i/I_{i-1}\simeq U_i\otimes_R \Cs_i, $ for some $ U_i\in R\proj.$ Furthermore, 
		\begin{enumerate}[(a)]
			\item If $\Ext_A^1(\Cs(\alpha), \Cs(\beta))\neq 0$, then $\alpha>\beta$.
			\item If $\Ext_A^i(\Cs(\alpha), \Cs(\beta))\neq 0$ for some $i>0$, then $\alpha>\beta$. In particular, $\Ext_A^i(\Cs(\alpha), \Cs(\alpha))=0, \ i>0$.
		\end{enumerate}
	\end{Prop}
	\begin{proof}
		$DM$ is a projective $A^{op}$-module (see for example \citep[Lemma 2.11]{CRUZ2022410}).  Recall that $(A^{op}, D\Cs(\l))$ is split highest weight category. There exists a filtration $
		0=P_{n+1}\subset P_n\subset\cdots\subset P_1=DM
		$ with \mbox{$P_i/P_{i+1}\simeq D\Cs_i\otimes_R U_i,$}  $1\leq i\leq n$ (see for example \citep[Proposition 4.13]{Rouquier2008} or \citep[Proposition B.0.5]{cruz2021cellular}). Applying $D$ yields the exact sequence \begin{align}
			0\rightarrow D(D\Cs_i\otimes_R U_i)\rightarrow DP_i\rightarrow DP_{i+1}\rightarrow 0.\label{eq88}
		\end{align}Note that $
		D(D\Cs_i\otimes_R U_i)\simeq \Hom_R(U_i, \Hom_R(D\Cs_i, R))\simeq \Hom_R(U_i, \Cs_i)\simeq DU_i\otimes_R \Cs_i.
		$ In particular, \linebreak$DP_n\simeq DU_n\otimes_R \Cs_n$ and $DP_1\simeq M$. Now by induction using at each step the filtration of $DP_{i+1}$ and the exact sequence $(\ref{eq88})$ we can construct a filtration of $DP_i$:
		$		0\subset I_{i}\subset I_{i+1}\subset \cdots \subset I_n=DP_i,$ satisfying $I_j/I_{j-1}\simeq DU_j\otimes_R \Cs_j$. 
		
		The second part follows from \citep[Proposition 4.13]{Rouquier2008} together with Proposition \ref{existenceofcostandardmodules} and Lemma 2.15 of \citep{CRUZ2022410}.
	\end{proof}As usual, denote by $\Cssim$ the set $\{\Cs(\l)\otimes_R U_\l\colon U_\l\in R\proj \}$.
	This means that the filtrations for $M\in \mathcal{F}(\Cssim)$ can be chosen so that the costandard modules with the lowest index appear at the bottom of the filtration (see also its dual version in \citep[Proposition B.0.6]{cruz2021cellular}). Analogously to the case of standard modules, we will use the notations
	$\Cssim_{\mu>\l}$ and $\Cssim_{\mu<\l}$.

	\begin{Theorem}\label{quasihereditaryintermsofcostandards}
		Let $A$ be a projective Noetherian $R$-algebra and let $\L$ be a poset. Then, there exists modules $\{\St(\l)\colon \l\in\L \}$ such that $(A, \{\Delta(\lambda)_{\lambda\in \Lambda}\})$ is split quasi-hereditary if and only if there exist modules $\{\Cs(\l)\colon \l\in\L \}$ satisfying the following properties:
		\begin{enumerate}[(i)]
			\item The modules  $\Cs(\l)\in A\m$ are projective over $R$ for every $\l\in \L$.
			\item Given $\alpha, \beta\in \L$, if $\Hom_A(\Cs(\alpha), \Cs(\beta))\neq 0$, then $\alpha\geq \beta$.
			\item $\End_A(\Cs(\l))\simeq R$, $\l\in\L$.
			\item For each $\l\in\L$, there exists an $(A, R)$-injective module which is projective as $R$-module $I(\l)$ together with an exact sequence $
			0\rightarrow \Cs(\l)\rightarrow I(\lambda)\rightarrow K(\l)\rightarrow 0, \quad K(\l)\in \mathcal{F}(\Cssim_{\mu>\l}).
			$
			\item $DA_A\in \add\left( \sumSt I(\l)\right) $, where $A_A$ denotes the right regular $A$-module.
		\end{enumerate}
	\end{Theorem}
	\begin{proof}
		Assume that there exists modules $\{\St(\l)\colon \l\in\L \}$ such that $(A, \{\Delta(\lambda)_{\lambda\in \Lambda}\})$ is split quasi-hereditary. By Proposition \ref{existenceofcostandardmodules}, $(A^{op}, \{D\Cs(\lambda)_{\lambda\in \Lambda}\})$ is a split quasi-hereditary algebra. Applying \citep[Proposition 2.2]{CRUZ2022410} to axioms (ii) and (iii) of Definition \ref{qhdef} we obtain (ii) and (iii). Applying $D$ to the remaining axioms, the remaining conditions follow (see also Proposition \ref{relativeinjectiveshavefiltrations}).
		
		The previous argument also shows that conditions (i), (ii), (iii), (iv) and (v) are equivalent to $(A^{op}, \{D\Cs(\lambda)_{\lambda\in \Lambda}\})$ being a quasi-hereditary algebra. Hence, the result follows applying Proposition \ref{existenceofcostandardmodules}(i) to $(A^{op}, \{D\Cs(\lambda)_{\lambda\in \Lambda}\})$.
	\end{proof}

	\begin{Lemma}	\citep[Lemma 4.21]{Rouquier2008}\label{extvanishesforfiltratedmodules}Let $(A, \{\Delta(\lambda)_{\lambda\in \Lambda}\})$ be a split quasi-hereditary algebra. 
		Let $M, N\in A\m$. The following holds.
		\begin{enumerate}[(a)]
			\item If $M\in \mathcal{F}(\Stsim)$, then $\Ext_A^i(M, \Cs(\l))=0, \ i>0$.
			\item If $N\in \mathcal{F}(\Cssim)$, then $\Ext_A^i(\St(\l), N)=0, \ i>0$.
			\item If $M\in \mathcal{F}(\Stsim)$ and $N\in \mathcal{F}(\Cssim)$, then $\Ext_A^i(M, N)=0, \ i>0$.
			\item If  $M\in \mathcal{F}(\Stsim)$ or $M\in \mathcal{F}(\Cssim)$, then $M\in R\proj$.
		\end{enumerate} 
	\end{Lemma}
	\begin{proof}
		Observe that for $i>0$ and every $\beta, \l\in \L$, $U\in R\proj$ the $R$-module ${\Ext_A^i(\St(\beta)\otimes_R U, \Cs(\l))}$ is in the additive closure of $\Ext_A^i(\St(\beta), \Cs(\l))$. Hence, $\Ext_A^i(\St(\beta)\otimes_R U, \Cs(\l))=0$ by Proposition \ref{existenceofcostandardmodules}.
		
		Let $M\in \mathcal{F}(\Stsim)$.  There is a filtration \begin{align}
			0=P_{n+1}\subset P_n\subset \cdots \subset P_1=M, \quad \text{ with } P_i/P_{i+1}\simeq \St_i\otimes_R U_i.\label{eq89}
		\end{align} Let $\l\in \L$. Applying $\Hom_A(-, \Cs(\l))$ to the exact sequence of $P_i$ yields the exact sequence
		\begin{align*}
			0=\Ext_A^j(\St_i\otimes_R U_i)\rightarrow \Ext_A^j(P_i, \Cs(\l))\rightarrow \Ext_A^j(P_{i+1}, \Cs(\l))\rightarrow \Ext_A^{j+1}(\St_i\otimes_R U_i, \Cs(\l))=0, \ \forall j>1.
		\end{align*} We conclude, for $j>1$,$
		\Ext_A^j(P_i, \Cs(\l))\simeq \Ext_A^j(P_{i+1}, \Cs(\l))\simeq \Ext_A^j(P_n, \Cs(\l))=\Ext_A^j(\St_n, \Cs(\l))=0.
		$ So $(a)$ holds. Since $\St_i\otimes_R U_i$ is projective over $R$ all the exact sequences
		$0\rightarrow P_{i+1}\rightarrow P_i\rightarrow \St_i\otimes_R U_i\rightarrow 0.
		$   are split over $R$. Thus, every $P_i$ is projective over $R$. In particular, $M\in R\proj$. The argument is analogous for $M\in \mathcal{F}(\Cssim)$.
		
		The proof of $b)$ is analogous now applying the functor $\Hom_A(\St(\l), -)$ to the exact sequences given by a filtration of $N\in \mathcal{F}(\Cssim)$.
		Let $N\in \mathcal{F}(\Cssim)$. Applying $\Hom_A(-, N)$ to the exact sequences of the filtration $(\ref{eq89})$ we get the isomorphism
		$
		\Ext_A^{j>0}(P_i, N)\simeq  \Ext_A^{j>0}(P_{i+1}, N).
		$ 
		
		Therefore, $
		0=\Ext_A^{j>0}(\St\otimes_R U_n, N)=\Ext_A^{j>0}(P_n, N)\simeq \Ext_A^{j>0}(P_1, N)=\Ext_A^{j>0}(M, N). 
		$
		Assume $M\in \mathcal{F}(\Stsim)$. 
	\end{proof}

	\section{On Ext-projective and Ext-injective objects of $\mathcal{F}(\Stsim)$}\label{On Ext-projective and Ext-injective objects of}

	The following result is Lemma 4.21 of \cite{Rouquier2008}. For quasi-hereditary algebras over fields, there are many proofs of this result in the literature (see for example \cite{MR1128706}). However, for quasi-hereditary algebras over commutative Noetherian rings as far as the author knows this result can only be found in \cite{Rouquier2008}. We present a different approach than the one used in \cite{Rouquier2008}, also because it is not clear to the author why $M/M_0$ is projective over $R$ using Rouquier's approach.

	\begin{Theorem}\label{filtrationsintermsofext}
		Let $(A, \{\Delta(\lambda)_{\lambda\in \Lambda}\})$ be a split quasi-hereditary algebra. 
		Let $M\in R\proj\cap A\m$. 
		\begin{enumerate}
			\item If $\Ext_A^1(M, \Cs(\l))=0, \ \forall \l\in\L$, then $M\in \mathcal{F}(\Stsim)$.
			\item If $\Ext_A^1(\St(\l), M)=0, \ \forall \l\in\L$, then $M\in \mathcal{F}(\Cssim)$.
		\end{enumerate}
	\end{Theorem}
	\begin{proof} Assume that $\Ext_A^1(M, \Cs(\l))=0$ for some $M\in R\proj\cap A\m$. By induction on the size of filtrations of modules in $\mathcal{F}(\Cssim)$ we deduce that $\Ext_A^1(M, N)=0$ for every $N\in \mathcal{F}(\Cssim)$.
		Let $\l\in \L$ be maximal. Thus, $\St(\l)\in \mathcal{M}(A)$. Recall that $\tau_{\St(\l), A}$ is a left and right $(A, R)$-monomorphism (see for example \citep[Lemma A.1.8., Proposition A.1.4.]{cruz2021cellular}). 
		Analogously, we can consider the left $A$-homomorphism 
		$\tau_{\St(\l), M}\colon \St(\l)\otimes_R \Hom_A(\St(\l), M)\rightarrow M$. If $M\in \mathcal{F}(\Stsim)$, then it is possible to construct a filtration with $\St(\l)\otimes_R U_\l$ appearing at the bottom, where $U_\l$ is a projective $R$-module (possibly the zero module). Therefore, we want to show that $\tau_{\St(\l), M}$ is an $(A, R)$-monomorphism. If we show in addition that its cokernel belongs to $\mathcal{F}(\Stsim)$, then we are done.

		\textit{Claim A.} We can relate $\tau_{\St(\l), A}\otimes_A M$ and $\tau_{\St(\l), M}$ through the following commutative diagram:
		\begin{equation}
			\begin{tikzcd}
				\St(\l)\otimes_R \Hom_A(\St(\l), A)\otimes_A M\arrow[r, "\tau_{\St(\l) {,} A}\otimes_A M", swap, outer sep=0.75ex] \arrow[d, "\St(\l)\otimes_R \psi", "\simeq"'] & A\otimes_A M \arrow[d, "\mu_M", "\simeq"'] \\
				\St(\l)\otimes_R \Hom_A(\St(\l), M)\arrow[r, "\tau_{\St(\l){,} M}"] & M
			\end{tikzcd},\label{eqqqh4}
		\end{equation}where $\mu_M$ is the multiplication map and $\psi$ is the natural isomorphism provided by the isomorphism of functors $\Hom_A(\St(\l), A) \otimes_A -\simeq \Hom_A(\St(\l), -)\colon A\m\rightarrow R\m$.
		In fact,
		\begin{align}
			\tau_{\St(\l), M}\St(\l)\otimes_R \psi(l\otimes f\otimes m)&=\tau_{\St(\l), M}(l\otimes \psi(f\otimes m))=\psi(f\otimes m)(l)=f(l)m \\
			\mu_M\circ \tau_{\St(\l), A}\otimes_A M(l\otimes f\otimes m)&=\mu_M(f(l)\otimes m)=f(l)m, \ l\in \St(\l), f\in \Hom_A(\St(\l), A), m\in M.
		\end{align}

		\textit{Claim B.} There are isomorphisms $\delta$ and $\theta$ making the following diagram commutative
		\begin{equation}
			\begin{tikzcd}
				\St(\l)\otimes_R \Hom_A(\St(\l), A)\otimes_A M\arrow[r, "\tau_{\St(\l) {,} A}\otimes_A M"] \arrow[d, "\delta", "\simeq"'] & A\otimes_A M \arrow[d, "\theta", "\simeq"'] \\
				D\Hom_A(M, D(\St(\l)\otimes_R \Hom_A(\St(\l), A)) \arrow[r, "D\Hom_A(M{,} D\tau_{\St(\l) {,} A})", outer sep=0.75ex]& D\Hom_A(M, DA)
			\end{tikzcd}.\label{eqqqh3}
		\end{equation}
		
		Note that by Tensor-Hom adjunction $D\Hom_A(M, DA)\simeq DDM$.
		Hence, the map \linebreak$\theta\in \Hom_A(A\otimes_A M, D\Hom_A(M, DA))$ given by $\theta(a\otimes m)(g)=g(am)(1_A)$ is an isomorphism. Further, as left $A$-modules,
		\begin{align*}
			D\Hom_A(M, D(\St(\l)\otimes_R \Hom_A(\St(\l), A))&\simeq DD(\St(\l)\otimes_R \Hom_A(\St(\l), A)\otimes_A M)\\&\simeq \St(\l)\otimes_R \Hom_A(\St(\l), A)\otimes_A M.
		\end{align*} Denote by $\delta\in \Hom_A(\St(\l)\otimes_R \Hom_A(\St(\l), A)\otimes_A M, D\Hom_A(M, D(\St(\l)\otimes_R \Hom_A(\St(\l), A)))$ this isomorphism. Explicitly, for every $l\in \St(\l), f\in \Hom_A(\St(\l), A), \ m\in M,$  \begin{align*}
			\delta(l\otimes f\otimes m)(g)=g(m)(l\otimes f), \ g\in \Hom_A(M, D(\St(\l)\otimes_R \Hom_A(\St(\l), A)).
		\end{align*}
		Let $l\otimes f \otimes m\in \St(\l)\otimes_R \Hom_A(\St(\l), A)\otimes_A M$, $g\in \Hom_A(M, DA)$. Then,
		\begin{align*}
			D\Hom_A(M, D\tau_{\St(\l), A})\circ \delta(l\otimes f\otimes m)(g)&=\delta(l\otimes f
			\otimes m)\Hom_A(M, D\tau_{\St(\l), A}) (g)\\=\delta(l\otimes f\otimes m)(D\tau_{\St(\l), A} \circ g)=D\tau_{\St(\l), A} g(m) (l\otimes f) &= g(m)\circ \tau_{\St(\l), A}(l\otimes f)=g(m)(f(l)).
		\end{align*}
		On the other hand,
		\begin{align*}
			\theta \tau_{\St(\l), A}\otimes_A M(l\otimes f\otimes m)(g)&=\theta(\tau_{\St(\l), A}(l\otimes f)\otimes m)(g)=\theta(f(l)\otimes m)(g)\\&=g(f(l)m)(1_A)=(f(l)\cdot g(m))(1_A)=g(m)(1_Af(l))=g(m)(f(l)).
		\end{align*}
		This shows that the diagram (\ref{eqqqh3}) is commutative and Claim B follows.

		\textit{Claim C.} The map $D\Hom_A(M, D\tau_{\St(\l), A})$ is a left $(A, R)$-monomorphism.
		
		The cokernel of the right $(A, R)$-monomorphism $\tau_{\St(\l), A}$ is $A/J\in R\proj$ where $J$ is the image of $\tau_{\St(\l), A}$, and therefore $J$ is a split heredity ideal. Hence, $A/J$ belongs to $\mathcal{F}(\Stsim_{op})$. Thus, $D(A/J)$ belongs to $\mathcal{F}(\Cssim)$. So, $D=\Hom_R(-, R)$ induces left $(A, R)$-exact sequence
		\begin{align}
			0\rightarrow D(A/J)\rightarrow DA \xrightarrow{D\tau_{\St(\l), A}} D(\St(\l)\otimes_R \Hom_A(\St(\l), A))\rightarrow 0. \label{eqqqh1}
		\end{align}
		Applying $\Hom_A(M, -)$ yields the exact sequence
		\begin{align*}
			0\rightarrow \Hom_A(M, D(A/J))\rightarrow \Hom_A(M, DA)\xrightarrow{\Hom_A(M, D\tau_{\St(\l), A})} \Hom_A(M, D(L\otimes_R \Hom_A(L, A)) \rightarrow 0, 
		\end{align*} because $\Ext_A^1(M, D(A/J))=0.$
		Due to $\St(\l)\in A\proj$, by Tensor-Hom adjunction, we have \begin{align*}
			\Hom_A(M, D(\St(\l)\otimes_R \Hom_A(\St(\l), A))&\simeq \Hom_R(\St(\l)\otimes_R \Hom_A(\St(\l), A)\otimes_A M, R)\\&\simeq \Hom_R(\St(\l)\otimes_R \Hom_A(\St(\l), M), R)\in R\proj.
		\end{align*} This shows that the right $A$-homomorphism $\Hom_A(M, D\tau_{\St(\l), A})$ is an $(A, R)$-epimorphism. Therefore, \linebreak$D\Hom_A(M, D\tau_{\St(\l), A})$ is a left $(A, R)$-monomorphism.
		
		Combining Claims A, B and C, we obtain that $\tau_{\St(\l), M}$ is a left $(A, R)$-monomorphism. 
		
		Let $X$ be the cokernel of $\tau_{\St(\l), M}$. In particular, $X\in R\proj$ and the exact sequence
		\begin{align}
			0\rightarrow \St(\l)\otimes_R \Hom_A(\St(\l), M)\xrightarrow{\tau_{\St(\l), M}} M \rightarrow X \rightarrow 0 \label{eqqqh5}
		\end{align}is $(A, R)$-exact. Recall that $U_\l:=\Hom_A(\St(\l), M)\in R\proj$. It remains to show that $X\in \mathcal{F}(\Stsim)$. The exactness of $\Hom_A(\St(\l), -)$ implies that the map $\Hom_A(\St(\l), \tau_{\St(\l), M})$ is injective. We claim that it is also surjective. Let $h\in \Hom_A(\St(\l), M)$. Then, for any $x\in \St(\l)$,
		\begin{align}
			h(x)=\tau_{\St(\l), M}(x\otimes h)=\tau_{\St(\l), M}\circ (-\otimes h) (x),
		\end{align}where $-\otimes h\in \Hom_A(\St(\l), \St(\l)\otimes_R \Hom_A(\St(\l), M)).$ Consequently, $\Hom_A(\St(\l), X)=0$ and so $X\in A/J\m\cap R\proj$ (see for example \citep[Corollary A.1.13]{cruz2021cellular}). 
		
		We will proceed by induction on $|\L|$ to show that every $Y\in A\m\cap R\proj$ satisfying $\Ext_A^1(Y, \Cs(\l))=0$ for every $\l\in \L$ belongs to $\mathcal{F}(\Stsim)$.
		
		If $|\L|=1$, then $A/J\m$ is the zero category, and thus $X=0$. By (\ref{eqqqh5}) $M\in \mathcal{F}(\Stsim)$. Assume that the result holds for split quasi-hereditary algebras with $|\L|<n$ for some $n>1$. Assume that $|\L|=n$. By Proposition \ref{existenceofcostandardmodules}, $
		\Hom_A(\St(\l)\otimes_R U_\l, \Cs(\alpha))\simeq \Hom_R(U_\l, \Hom_A(\St(\l), \Cs(\alpha))=0, \ \alpha\neq \l.
		$ Let $\alpha\in \L$ distinct of $\l$. The functor $\Hom_A(-, \Cs(\alpha))$ induces the long exact sequence
		\begin{align}
			0=\Hom_A(\St(\l)\otimes_R U_\l, \Cs(\alpha))\rightarrow \Ext_A^1(X, \Cs(\alpha))\rightarrow \Ext_A^1(M, \Cs(\alpha))=0.
		\end{align}By induction, $X\in \mathcal{F}(\Stsim_{\alpha\neq \lambda})$. By (\ref{eqqqh5}) $M\in \mathcal{F}(\Stsim)$.
		Now assume that $\Ext_A^1(\St(\mu), M)=0$ for every $\mu\in\L$ and $M\in A\m\cap R\proj$. Since $\St(\mu), M\in R\proj$, \citep[Lemma 2.15]{CRUZ2022410}  yields that $\Ext_{A^{op}}^1(DM, D\St(\mu))=0$, $\mu\in\L.$ As $\{D\St(\mu)\}$ are costandard modules of $A^{op}$, we obtain by statement $1.$ that $DM\in \mathcal{F}_{A^{op}}(D\Cssim)$. Therefore, \mbox{$M\in \mathcal{F}(\Cssim)$}.
	\end{proof}
	So, it follows that  $\mathcal{F}(\Stsim)$ is a resolving subcategory of $A\m\cap R\proj$, as in the classical case. That is,  $\mathcal{F}(\Stsim)$ contains all projective $A$-modules, it is closed under extensions, under direct summands and it is also closed under kernels of epimorphisms.
	
	By Lemma \ref{extvanishesforfiltratedmodules}, we see that the condition of $M\in R\proj$ cannot be dropped in Theorem \ref{filtrationsintermsofext}. A trivial example to check this situation is the split quasi-hereditary algebra $R$ for some commutative Noetherian ring $R$ with positive global dimension and trivial Picard group. Then, $\Cs=\St=R$, and therefore $\mathcal{F}(\Stsim)=R\proj$ while $\{K\in R\m\colon \Ext_R^1(R, K)=0 \}=R\m$. 
	
	Also from the proof of Theorem \ref{filtrationsintermsofext} it follows
	\begin{Cor}
		Let $(A, \{\Delta(\lambda)_{\lambda\in \Lambda}\})$ be a split quasi-hereditary algebra. Let $\l\in \L$ be maximal.   The map $\tau_{\St(\l), M}$ is an $(A, R)$-monomorphism for all $M\in \mathcal{F}(\Stsim)$.
	\end{Cor}

	\begin{Lemma}\label{projectiveandinjectiveinsidefiltrations}
		Let $(A, \{\Delta(\lambda)_{\lambda\in \Lambda}\})$ be a split quasi-hereditary algebra.  
		\begin{enumerate}
			\item Let $M\in \mathcal{F}(\Stsim)$. If $\Ext^1_A(M, \sumSt \St(\l))=0$, then $M$ is projective over $A$.
			\item Let $N\in \mathcal{F}(\Cssim)$. If $\Ext^1_A(\sumSt \Cs(\l), N)=0$, then $N$ is $(A, R)$-injective.
		\end{enumerate} 
	\end{Lemma}
	\begin{proof} See for example 	\citep[Lemma 4.22]{Rouquier2008}. 
		
		Consider the $(A, R)$-exact sequence $0\rightarrow N\rightarrow \Hom_R(A, N)\rightarrow X\rightarrow 0$. Applying $\Hom_A(\sumSt \St(\l), -)$ yields $\Ext_A^1(\sumSt \St(\l), X)=0$. By Theorem \ref{filtrationsintermsofext}, $X\in \mathcal{F}(\Cssim)$. By assumption, $\Ext_A^1(X, N)=0$. Hence, $N$ is an $A$-summand of $\Hom_R(A, N)$ and consequently, it is $(A, R)$-injective. 
	\end{proof}
	This lemma says that the $\Ext$-projective objects for $\mathcal{F}(\Stsim)$ belonging to $\mathcal{F}(\Stsim)$ are exactly the projective $A$-modules.  
	\subsection{Characteristic tilting modules} \label{Tilting modules}
	Characteristic tilting modules of finite-dimensional quasi-hereditary algebras are fundamental objects in order to obtain information about simple modules, and therefore about the structure of $A\m$. Their summands are known as (partial) tilting modules. In the Noetherian case, the (partial) tilting modules behave very similarly to the classical case. 
	Previous uses of partial tilting modules for split quasi-hereditary algebras over commutative Noetherian rings can be found in \citep{Rouquier2008}, \citep[III. 4]{zbMATH01527053}, \citep{zbMATH06751586}. Partial tilting modules for $S_\mathbb{Z}(n, d)$, $n\geq d$ were studied in \citep[section 3]{zbMATH00549737}.
	
	\subsubsection{Existence of characteristic tilting modules}
	
	\begin{Def}\label{tilting}
		A module $T\in A\m$ is called \textbf{(partial) tilting} if $T\in \mathcal{F}(\Stsim)\cap \mathcal{F}(\Cssim)$.
	\end{Def}
	
	A starting point to construct them is the following result.
	\begin{Prop}\label{tiltingconstruction}
		Let	 $(A, \{\Delta(\lambda)_{\lambda\in \Lambda}\})$ be a split quasi-hereditary algebra. The following assertions hold.
		\begin{enumerate}[(a)]
			\item Let $M\in \mathcal{F}(\Stsim)$. There is a partial tilting module $T_M$ and a monomorphism $i_M\colon M\rightarrow T_M$ such that $\coker i_M\in \mathcal{F}(\St)$.
			\item  Let $\l\in\L$. There are exact sequences and a partial tilting module $T(\l)$
			\begin{align}
				0\rightarrow \St(\l)\rightarrow T(\l)\rightarrow X(\l)\rightarrow 0 \label{eq121}\\
				0\rightarrow Y(\l)\rightarrow T(\l)\rightarrow \Cs(\l)\rightarrow 0 \label{eq121b},
			\end{align} where $X(\l)\in \mathcal{F}(\St_{\mu<\l}), \ Y(\l)\in \mathcal{F}(\Cssim_{\mu<\l})$. 
		\end{enumerate}
	\end{Prop}
	\begin{proof} We follow the same argument given in \citep[Proposition 4.26]{Rouquier2008}. There is a subtle difference, here, we argue that $T$ can be constructed using only free modules in the filtration.
		Let $M\in \mathcal{F}(\Stsim)$. Fix $\L\rightarrow \{1, \ldots, t\}$, $\l\mapsto i_\l$ an increasing bijection and set $\St_{i_\l}:=\St(\l)$. We construct by induction an object $T_M$ with a filtration \begin{align}
			0=T_{n+1}\subset M=T_n\subset \cdots \subset T_0=T_M, \quad T_{i-1}/T_i\simeq \St_i\otimes_R U_i, \ U_i\in \free.
		\end{align}
		The case $n=1$ follows from $\Ext_A^1(\St_1, \St_1)=0$ and the fact that $\St_1\otimes_R U_1$ is a direct sum of copies of $\St_1$ for any $U_1\in R\proj$. Assume $n>1$. Assume $T_i$ is defined for some $i$, $2\leq i\leq n$. We shall construct $T_{i-1}$. Let $U_i$ be a free $R$-module defined by the following map $U_i\xrightarrow{\pi} \Ext_A^1(\St_i, T_i)$ being surjective. Consider the extension
		\begin{align}
			0\rightarrow T_i\rightarrow X\rightarrow \St_i\otimes_R U_i\rightarrow 0\label{eq125}
		\end{align}corresponding to $\pi$ via the isomorphism $\Hom_R(U_i, \Ext_A^1(\St_i, T_i))\rightarrow \Ext_A^1(\St_i\otimes_R U_i, T_i)$ (see for example \citep[Lemma B.0.2]{cruz2021cellular}). In particular, \mbox{$\Ext_A^1(\St_i, X)=0$.} Define $T_{i-1}=X$. Since $T_i\in \mathcal{F}(\St_{j>i})$ we obtain that $T_{i-1}\in \mathcal{F}(\St_{j\geq i})$. On the other hand, for $j>i$, applying $\Hom_A(\St_j, -)$ to (\ref{eq125}) yields \begin{align}
			\Hom_A(\St_j, \St_i\otimes_R U_i)\rightarrow \Ext_A^1(\St_j, T_i)\rightarrow \Ext_A^1(\St_j, X)\rightarrow \Ext_A^1(\St_j, \St_i\otimes_R U_i)=0.
		\end{align}We can assume by induction that $\Ext_A^1(\St_j, T_i)=0$ for $j>i$. Hence, $\Ext_A^1(\St_j, T_{i-1})=0$ for $j>i$. Hence, $\Ext_A^1(\St_j, T_{i-1})=0$ for all $j\geq i$. Hence, by induction, we obtain a module $T_M\in \mathcal{F}(\St)$ with $\Ext_A^1(\St_j, T_M)=0$ for all $j$ and satisfying $T_M/M\in \mathcal{F}(\St)$. By Theorem \ref{filtrationsintermsofext}, $T_M$ is partial tilting.
		
		Now consider $M=\St(\l)=\St_i=T_i$. Notice that we can start the construction of $T$ at $i$ since for $j>i$ we have $\Ext_A^1(\St_j, \St_i)=0$. Applying the previous construction we have a filtration \begin{align}
			0\subset \St_i=T_i\subset T_{i-1}\subset \cdots \subset T_0=T(i), \ T_{j-1}/T_j\simeq \St_j\otimes_R F_j, \ F_j\in R\free 
		\end{align} with $T(i)$  being a partial tilting module.  
		Since $T(i)\in \mathcal{F}(\Cssim)$ there exists a filtration $
		0=I_0\subset I_1\subset \cdots \subset I_n=T(i)$ with $I_j/I_{j-1}\simeq \Cs_j\otimes_R U_j, \ 1\leq j \leq n.
		$ Consider the exact sequences \begin{align}
			0\rightarrow I_{j-1}\rightarrow I_j\rightarrow \Cs_j\otimes_R U_j\rightarrow 0, \ 1\leq j\leq n.\label{eqqh128}
		\end{align} It is enough to show that $U_j=0$ for $j>i$ and $U_i\simeq R$. Let $1\leq k\leq n$. Applying the functor $\Hom_A(\St_k, -)$ we obtain the exact sequences
		\begin{align}
			0\rightarrow \Hom_A(\St_k, I_{j-1})\rightarrow \Hom_A(\St_k, I_j)\rightarrow \Hom_A(\St_k, \Cs_j\otimes_R  U_j)\rightarrow \Ext_A^1(\St_k, I_{j-1})=0.
		\end{align}Hence, for $k\neq j$ we obtain \begin{align}
			\Hom_A(\St_k, I_{j-1})\simeq \Hom_A(\St_k, I_j).\label{eq131}
		\end{align} For $k=j$, the following is exact\begin{align}
			0\rightarrow \Hom_A(\St_j, I_{j-1})\rightarrow \Hom_A(\St_j, I_j)\rightarrow U_j\rightarrow 0, \ 1\leq j\leq n. \label{eq132}
		\end{align}
		Combining (\ref{eq131}) with (\ref{eq132}) we obtain the exact sequence
		\begin{align}
			0\rightarrow \Hom_A(\St_k, \Cs_1\otimes_R U_1)\rightarrow \Hom_A(\St_k, T(i))\rightarrow U_k\rightarrow 0, \ k>1. \label{eq133}
		\end{align}
		If $k>i$ then since $T(i)\in \mathcal{F}(\St_{j\leq i})$ we obtain $\Hom_A(\St_k, T(i))=0$. By (\ref{eq133}), $U_j=0$ for $j>i$.
		
		If $i=1$, it follows by (\ref{eqqh128}) and (\ref{eq131})\begin{align}
			U_1\simeq \Hom_A(\St_1, I_1)\simeq \Hom_A(\St_1, I_n)=\Hom_A(\St_1, T(1))=\Hom_A(\St_1, \St_1)\simeq R.
		\end{align} 
		Assume $i>1$. By (\ref{eq133}), $U_i=\Hom_A(\St_i, T(i))$.  Finally, observe that $\Hom_A(\St_i, T(i))=R$. In fact, using the exact sequence constructed $0\rightarrow\St_i\rightarrow T(i)\rightarrow X(i)\rightarrow 0$, every morphism in $\Hom_A(\St_i, T(i))$ factors through $\St_i$ since $X(i)\in \mathcal{F}(\St_{j<i})$. 
	\end{proof}
	
	We say that $T=\sumSt T(\l)$ is a \textbf{characteristic tilting module} of $(A, \{\Delta(\lambda)_{\lambda\in \Lambda}\})$ (or just of $A$ when there is no confusion on the underlying quasi-hereditary structure of $A$), if each $T(\l)$ is a partial tilting with exact sequences as in Theorem \ref{tiltingconstruction}, where we can relax the conditions on $X(\l)$ and $Y(\l)$ to $X(\l)\in \mathcal{F}(\Stsim_{\mu<\l})$ and $Y(\l)\in \mathcal{F}(\Cssim_{\mu<\l})$.
	As we will see, a characteristic tilting module is a full tilting module justifying the modules $T(\l)$ being called (partial) tilting.

	In practice, the short exact sequences (\ref{tiltingconstruction}) provide a way for determining the  (partial) tilting modules. But, as we will see next, these short exact sequences are also approximations. Recall that given a subcategory of $A\m$, $\mathcal{C}$, and a module $M\in A\m$, a \textbf{left $\mathcal{C}$-approximation} of $M$ (if it exists) is a map $f\colon M\rightarrow N$ so that the induced map $\Hom_A(f, C)$ is surjective for any $C\in \mathcal{C}$ with $N\in \mathcal{C}$.

	\begin{Prop}\label{approximationstilting}
		Let $(A, \{\Delta(\lambda)_{\lambda\in \Lambda}\})$ be a split quasi-hereditary algebra. Let $\l\in\L$.
		
		The homomorphism $\St(\l)\hookrightarrow T(\l)$ constructed in Proposition \ref{tiltingconstruction} is an injective left $\mathcal{F}(\Cssim)$-approximation of $\St(\l)$. The homomorphism $T(\l)\rightarrow \Cs(\l)$ constructed in Proposition \ref{tiltingconstruction} is a surjective right $\mathcal{F}(\Stsim)$-approximation of $\Cs(\l)$.
	\end{Prop}
	\begin{proof}
		Let $X\in \mathcal{F}(\Cssim)$. Applying $\Hom_A(-, X)$ to (\ref{eq121}) yields the exact sequence
		\begin{align}
			0\rightarrow \Hom_A(X(\l), X)\rightarrow \Hom_A(T(\l), X)\rightarrow \Hom_A(\St(\l), X)\rightarrow \Ext_A^1(X(\l), X).
		\end{align}By Lemma \ref{extvanishesforfiltratedmodules}, $\Ext_A^1(X(\l), X)=0$  since $X(\l)\in \mathcal{F}(\Stsim)$. 
		Thus, $\Hom_A(T(\l), X)\rightarrow \Hom_A(\St(\l), X)$ is surjective. 
		
		Let $Y\in \mathcal{F}(\Stsim)$. Applying $\Hom_A(Y, -)$ to (\ref{eq121b}) yields that the map ${\Hom_A(Y, T(\l))\rightarrow \Hom_A(Y, \Cs(\l))}$ is surjective.
	\end{proof}
	
	There is naturally a version of Corollary \ref{dualofstandardiscostandard} and Proposition \ref{existenceofcostandardmodules} for partial tilting modules.
	
	\begin{Lemma}\label{characteristictiltingforaop}
		Let $(A, \{\Delta(\lambda)_{\lambda\in \Lambda}\})$ be a split quasi-hereditary algebra and let $T$ be a partial tilting module. Then, $DT$ is a partial tilting module in the split highest weight category $(A^{op}, \{{D\Cs(\l)}_{\lambda\in \Lambda}\})$. Moreover, if $T$ is a characteristic tilting module over $A$, then $DT$ is a characteristic tilting module over $A^{op}$.
	\end{Lemma}
	\begin{proof}
		By Theorem \ref{filtrationsintermsofext}, $DT\in \mathcal{F}(D\Stsim)\cap \mathcal{F}(D\Cssim)$. Assume that $T$ is a characteristic tilting module. The exact sequences (\ref{eq121}) and (\ref{eq121b}) are $(A, R)$-exact since $X(\l), \Cs(\l)\in R\proj$. Applying $D$, it follows that $DT$ is a characteristic tilting module over $A^{op}$.
	\end{proof}
	
	\subsubsection{Characterizations of $\mathcal{F}(\Stsim)$ and $\mathcal{F}(\Cssim)$ in terms of characteristic tilting modules}
	
	Let $\mathcal{C}$ be a subcategory of $A\m$. In the following, we will denote by $\widehat{\mathcal{C}}$ the subcategory of $A\m$ whose modules $M$ fit into an exact sequence $0\rightarrow C_t\rightarrow C_{t-1}\rightarrow \cdots\rightarrow C_0\rightarrow M\rightarrow 0$ with all $C_i\in \mathcal{C}$. Dually, we will consider the subcategory $\widecheck{\mathcal{C}}$.
	
	\begin{Lemma}\label{lemmaforresultiontilting}
		Let $(A, \{\Delta(\lambda)_{\lambda\in \Lambda}\})$ be a split quasi-hereditary algebra. Assume that $T$ is a characteristic tilting module.	Let $X, Z\in \widehat{\add T}\cap R\proj$. Assume there is an exact sequence $
		0\rightarrow X\xrightarrow{k}Y\xrightarrow{\pi}Z\rightarrow 0.
		$ Then, $Y\in \widehat{\add T}\cap R\proj$.
	\end{Lemma}
	\begin{proof}
		Consider the following diagram with exact rows and columns
		\begin{equation*}
			\begin{tikzcd}
				0\arrow[r] & X \arrow[r, "k"] &Y\arrow[r, "\pi"] &Z\arrow[r]& 0\\
				0\arrow[r] & T_0'\arrow[u, "p_0'", twoheadrightarrow] \arrow[r, "k_0"]& T_0'\bigoplus T_0''\arrow[r, "\pi_0"]& T_0'' \arrow[u, twoheadrightarrow, "p_0''"] \arrow[r]& 0\\
				& K_0'\arrow[u, hookrightarrow] & & K_0''\arrow[u, hookrightarrow]& 
			\end{tikzcd}
		\end{equation*} with $T_0', T_0''\in \mathcal{F}(\Stsim)\cap \mathcal{F}(\Cssim)$.
		Applying $\Hom_A(T_0'', -)$ to the top row yields
		\begin{align}
			0\rightarrow \Hom_A(T_0'', X)\rightarrow \Hom_A(T_0'', Y)\rightarrow \Hom_A(T_0'', Z)\rightarrow \Ext_A^1(T_0'', X)=0.
		\end{align}This is an immediate consequence of $T_0''\in \mathcal{F}(\Stsim)$ and $X\in \mathcal{F}(\Cssim)$. Hence, the map $p_0''$ lifts to $f\in \Hom_A(T_0'', Y)$ such that $p_0''=\pi\circ f$. Now consider $g\colon T_0'\bigoplus T_0''\rightarrow Y$, given by $g(x, y)=k\circ p_0'(x)+f(y), \ (x, y)\in T_0'\bigoplus T_0''$. Then, for $(x, y)\in T_0'\bigoplus T_0''$,
		\begin{align}
			g\circ k_0(x)=g(x, 0)&=k\circ p_0'(x)\\
			\pi\circ g(x, y)=\pi(k\circ p_0'(x)+f(y))=\pi\circ f(y)=p_0''(y)&=p_0''\circ \pi_0(x, y).
		\end{align}
		Hence, $g$ makes the previous diagram commutative. By Snake Lemma, $g$ is surjective. Define $K_0=\ker g$. $k_0|_{K_0'}\colon K_0'\rightarrow K_0$ is well defined and it is clearly a monomorphism since 
		$g\circ k_0(x)=k\circ p_0'(x)=k(0)=0, \ x\in K_0.$
		Now $\pi_0|_{K_0}\colon K_0\rightarrow K_0''$ is well defined since
		$
		p_0''\circ \pi_0|_{K_0}(x, y)=p_0''\circ \pi_0(x, y)=\pi\circ g(x, y)=0, \ (x, y)\in K_0.
		$
		Therefore, we have the commutative diagram with exact columns and the two top rows exact,
		\begin{equation*}
			\begin{tikzcd}
				0\arrow[r] & X \arrow[r, "k"] &Y\arrow[r, "\pi"] &Z\arrow[r]& 0\\
				0\arrow[r] & T_0'\arrow[u, twoheadrightarrow, "p_0'"] \arrow[r, "k_0"]& T_0'\bigoplus T_0''\arrow[u,twoheadrightarrow, "g"]\arrow[r, "\pi_0"]& T_0'' \arrow[u, twoheadrightarrow,"p_0''"] \arrow[r]& 0\\
				0\arrow[r]	& K_0'\arrow[u, hookrightarrow]\arrow[r, "k_0|_{K_0'}"] & K_0\arrow[u, hookrightarrow] \arrow[r, "\pi_0|_{K_0}"] & K_0''\arrow[u, hookrightarrow]\arrow[r]& 0 
			\end{tikzcd}.
		\end{equation*}
		Let $y\in K_0''$. Then,$
		\pi\circ g(0, y)=\pi\circ f(y)=p_0''(y)=0.$Thus, $g(0, y)=k(p_0'(t))=g\circ k_0(t)$ for some $t\in T_0'$. Hence, $(0, y)-k_0(t)\in K_0$ and its image under $\pi_0$ is $y$. Thus, $\pi_0|_{K_0}$ is surjective.
		Let $(x, y)\in \ker \pi|_{K_0}$. Then, $(x, y)\in K_0\cap \im k_0$, so there exists $z\in T_0'$ such that $k_0(z)=(x,y)$. Thus,
		$
		k\circ p_0'(z)=g\circ k_0(z)=0, 
		$ and consequently $p_0'(z)=0.$ Thus, $z\in K_0'$. So, the bottom row is also exact.
		
		Now continue with the construction with the bottom row. Note that both $K_0', K_0''$ have partial tilting resolutions by construction. After a finite number of steps either we must proceed with an exact sequence \begin{align}
			0\rightarrow K_t'\rightarrow K_t\rightarrow K_t''\rightarrow 0,
		\end{align}with  $K_t'\in \add T, K_t''\in \mathcal{F}(\Cssim)$ or  $T_{t+1}''=K_t''\in \add T, K_t'\in \mathcal{F}(\Cssim)$. In the first case, proceed one more step and we end up with $K_{t+1}\simeq K_{t+1}''$. So,\begin{align}
			0\rightarrow T_r''\rightarrow \cdots \rightarrow T_{t+2}'' \rightarrow T_{t+1}'\bigoplus T_{t+1}''\rightarrow \cdots \rightarrow T_0'\bigoplus T_0''\rightarrow Y\rightarrow 0
		\end{align}is a partial tilting resolution for $Y$. In the second case, $\Ext_A^1(K_t'', K_t')=0$, so it splits, that is $K_t\simeq K_t''\bigoplus K_t'$. Hence
		\begin{align}
			0\rightarrow T_r'\rightarrow \cdots \rightarrow T_{t+2}'\rightarrow T_{t+1}'\bigoplus T_{t+1}''\rightarrow \cdots \rightarrow T_0'\bigoplus T_0''\rightarrow Y\rightarrow 0
		\end{align} is a partial tilting resolution for $Y$. The assertion that $Y\in R\proj$ is clear.
	\end{proof}

	\begin{Theorem}\label{severalpropertiestilting}
		Let $(A, \{\Delta(\lambda)_{\lambda\in \Lambda}\})$ be a split quasi-hereditary algebra. Assume that $T$ is a characteristic tilting module. The following assertions hold true.
		\begin{enumerate}[(a)]
			\item $\mathcal{F}(\Stsim)=\{M\in A\m\cap R\proj\colon \Ext_A^{i>0}(M, T)=0 \}$.
			\item $\mathcal{F}(\Stsim)=\widecheck{\add T}$.
			\item $\mathcal{F}(\Cssim)=\{N\in A\m\cap R\proj\colon \Ext_A^{i>0}(T, N)=0 \}$.
			\item $\mathcal{F}(\Cssim)=\widehat{\add T}\cap R\proj$.
			\item $\add T =\mathcal{F}(\Stsim)\cap \mathcal{F}(\Cssim)$.
		\end{enumerate}
	\end{Theorem}
	\begin{proof}
		Let $M\in \mathcal{F}(\Stsim)$. As $T\in \mathcal{F}(\Cssim)$ then $\Ext_A^{i>0}(M, T)=0$ by Lemma \ref{extvanishesforfiltratedmodules}. 
		Conversely, assume that $\Ext_A^{i>0}(M, T)=0$. Then, $\prod_{\l\in\L} \Ext_A^{i>0}(M, T(\l))=0$ and by consequence for each $\l\in \L$, $\Ext_A^{i>0}(M, T(\l))=0$.
		We claim that $\Ext_A^{i>0}(M, \Cs(\l))=0$ for every $\l\in\L$. 
		If $\l$ is minimal, then $T(\l)=\Cs(\l)$, so there is nothing to show. Assume that, for all $\mu<\l$, $\Ext_A^{i>0}(M, \Cs(\mu))=0$. Then, $\Ext_A^i(M, X)=0$ for every $X\in \mathcal{F}(\Cssim_{\mu<\l})$, $i>0$. 
		Applying $\Hom_A(M, -)$ to (\ref{eq121b}) we obtain that $\Ext_A^1(M, \Cs(\l))\simeq \Ext_A^2(M, Y(\l))=0$. By induction, $\Ext_A^1(M, \Cs(\l))=0$ for all $\l\in\L$. By Proposition \ref{filtrationsintermsofext}, $M\in \mathcal{F}(\Stsim)$. Hence, (a) follows. By a symmetric argument, we obtain statement (c).
		
		We will now prove (d). Let $X\in \widehat{\add T}\cap R\proj$. So $X$ has a finite resolution by $\add T$ which splits as sequence of $R$-modules. As $\mathcal{F}(\Cssim)$ is closed under quotients of $(A, R)$-monomorphisms it follows that \mbox{$X\in \mathcal{F}(\Cssim)$.} 
		Now we will show that each costandard module $\Cs(\mu)$ has a partial tilting resolution. If $\l$ is minimal, then \mbox{$\St(\l)=T(\l)=\Cs(\l)$.} So, there is nothing to prove in such a case. Assume by induction that each $\Cs(\mu)$ with $\mu<\l$ has a resolution by partial tilting modules.
		By Lemma \ref{lemmaforresultiontilting}, every module in $\mathcal{F}(\Cs_{\mu<\l})$ has a finite partial tilting resolution. Hence $Y(\l)$, as in Proposition \ref{tiltingconstruction}, has a finite partial tilting resolution. Now using the exact sequence (\ref{eq121b}) and the partial tilting resolution for $Y(\l)$, it follows that $\Cs(\l)$ has a finite partial tilting resolution. Applying Lemma \ref{lemmaforresultiontilting}, it follows that any module in $\mathcal{F}(\Cssim)$ has a partial tilting resolution. So, (d) follows.
		
		We will now proceed to prove (b). Let $M\in \widecheck{\add T}$. Since $\mathcal{F}(\Stsim)$ is closed under kernels of epimorphisms it follows that $M\in \mathcal{F}(\Stsim)$.
		Conversely, assume that $M\in \mathcal{F}(\Stsim)$. Then $DM\in \mathcal{F}(D\Stsim)$. By (d), Corollary \ref{dualofstandardiscostandard} and Lemma \ref{characteristictiltingforaop}, $DM\in \widehat{\add DT}$. Since $M\in R\proj$, $M\simeq DDM\in \widecheck{\add DDT}=\widecheck{\add T}$.
		So, (b) follows.
		
		It remains to prove assertion (e). The inclusion $\add T\subset \mathcal{F}(\Stsim)\cap \mathcal{F}(\Cssim)$ was established in Proposition \ref{tiltingconstruction}. 
		Let $X\in \mathcal{F}(\Stsim)\cap \mathcal{F}(\Cssim)$. By (d), there exists an $(A, R)$-exact sequence $0\rightarrow L\rightarrow T_0\rightarrow X\rightarrow 0$ with $L\in \mathcal{F}(\Cssim)$ and $T_0\in \add T$. Further, since $\mathcal{F}(\Stsim)$ is closed under kernels of epimorphisms $L$ also belongs to $\mathcal{F}(\Stsim)$. In particular, $\Ext_A^1(X, L)=0$ and therefore $T_0\simeq X\oplus L$. Applying $\Hom_A(T, -)$ we obtain that $\Hom_A(T, X)$ is projective over $\End_A(T)^{op}$. Using projectivization together with $T\otimes_{\End_A(T)^{op}} \Hom_A(T, -)$ being left exact on $T_1\rightarrow T_0\rightarrow X\rightarrow 0$ (given by (d)) we obtain by diagram chasing that $X\in \add T$.
	\end{proof}
	
	We should remark that a characteristic tilting module $T$ is in fact a \textbf{full generalized tilting module}, that is, it has finite projective dimension over $A$, it has no self-extensions, that is, $\Ext_A^{i>0}(T, T)=0$ and by Theorem \ref{severalpropertiestilting}(b) there exists an exact sequence $0\rightarrow A\rightarrow T_0\rightarrow \cdots \rightarrow T_r\rightarrow 0$ where $T_i\in \add T$ for all $0\leq i\leq r$ for some $r\in \mathbb{N}$. In the integral setup, characteristic tilting modules are not necessarily unique but they are unique up to "multiplicities".
	
	\begin{Cor}\label{uniquenessofringeldual}
		Let $(A, \{\Delta(\lambda)_{\lambda\in \Lambda}\})$ be a split quasi-hereditary algebra. Assume there are modules $T(\l)$ and $Q(\l), \l\in \L$ with exact sequences
		\begin{align*}
			0\rightarrow \St(\l)\rightarrow T(\l)\rightarrow X(\l)\rightarrow 0\\
			0\rightarrow Y(\l)\rightarrow T(\l)\rightarrow \Cs(\l)\rightarrow 0\\
			0\rightarrow \St(\l)\rightarrow Q(\l)\rightarrow X'(\l)\rightarrow 0\\
			0\rightarrow Y'(\l)\rightarrow Q(\l)\rightarrow \Cs(\l)\rightarrow 0
		\end{align*} where $X(\l), X'(\l)\in \mathcal{F}(\Stsim_{\mu<\l})$ and $Y(\l), Y'(\l)\in \mathcal{F}(\Cssim_{\mu<\l})$. Let $T=\sumSt T(\l), \ Q=\sumSt Q(\l)$. Then, $\add T=\add Q$. Further, $\End_A(T)^{op}$ and $\End_A(Q)^{op}$ are Morita equivalent.
	\end{Cor}
	\begin{proof}By assumption, both $Q$ and $T$ are characteristic tilting modules of  $(A, \{\Delta(\lambda)_{\lambda\in \Lambda}\})$.
		By Theorem \ref{severalpropertiestilting}(e), $\add T=\mathcal{F}(\Stsim)\cap \mathcal{F}(\Cssim)=\add Q$. By projectivization, $\End_A(T)^{op}\proj\simeq \End_A(Q)^{op}\proj$. Therefore, $\End_A(T)^{op}$ and $\End_A(Q)^{op}$ are Morita equivalent.
	\end{proof}

	\section{Change of rings and filtrations of $\Hom(\mathcal{F}(\Stsim), \mathcal{F}(\Cssim))$} \label{Change of rings and filtrations of}
	We will see now that costandard modules and partial tilting modules behave well under change of ground rings.
	
	\begin{Prop}\label{costandardsundergroundringchange}
		Let $S$ be a commutative $R$-algebra and a Noetherian ring. 	Let $(A, \{\Delta(\lambda)_{\lambda\in \Lambda}\})$ be a split quasi-hereditary algebra. Then, the following assertions hold.
		\begin{enumerate}[(a)]
			\item $(S\otimes_R A, \{S\otimes_R \Delta(\lambda)_{\lambda\in \Lambda}\})$ has costandard modules $S\otimes_R \Cs(\l)\otimes_S U(\l)$ for some $U(\l)\in Pic(S)$. Moreover, if $S$ is flat over $R$, then the costandard modules can be written in the form $S\otimes_R \Cs(\l)$.
			\item Let $T=\bigoplus_{\l\in \L} T(\l)$ be a characteristic tilting module of $A$. Assume that $S$ is flat over $R$ or that $S$ has a trivial Picard group then $S\otimes_R T(\l)$ is a partial tilting module (it satisfies (\ref{eq121}) and (\ref{eq121b})) for $S\otimes_RA$ and $S\otimes_R T$ is a characteristic tilting module.
		\end{enumerate} 
	\end{Prop}
	\begin{proof}
		By Proposition \ref{existenceofcostandardmodules}, $(A^{op}, \{D\Cs(\lambda)_{\lambda\in \Lambda}\})$ is a split quasi-hereditary algebra. By \citep[Proposition 4.14]{Rouquier2008}, \mbox{$(S\otimes_R A^{op}, \{S\otimes_R D\Cs(\lambda)_{\lambda\in \Lambda}\})$} is a split quasi-hereditary algebra. Now note that $(S\otimes_R A)^{op}=S\otimes_R A^{op}$, since $S$ is a commutative ring. Moreover,
		\begin{align}
			S\otimes_R D\Cs(\l)\simeq \Hom_{S\otimes_R R}(S\otimes_R \Cs(\l), S\otimes_R R)\simeq \Hom_S(S\otimes_R \Cs(\l), S).
		\end{align}So, $S\otimes_R \Cs(\l)\otimes_S U_\l$, for a fixed $U_\l\in Pic(S)$, is a costandard module of $S\otimes_R A$ by Proposition \ref{existenceofcostandardmodules}. Now assume that $S$ is a flat $R$-algebra. Then, \begin{align}
			\Ext_{S\otimes_R A}^j(S\otimes_R \St(\l), S\otimes_R \Cs(\beta))\simeq S\otimes_R \Ext_A^j(\St(\l), \Cs(\beta))\simeq \begin{cases}
				S\otimes_R R & \text{ if } \l=\beta, \ i=0\\
				0 & \text{ otherwise }
			\end{cases}.
		\end{align}
		By the uniqueness, $S\otimes_R \Cs(\l)$ are costandard modules of $S\otimes_R A$. 
		
		Assume that either $S$ is an $R$-flat or $S$ has trivial Picard group. Then, by $(b)$ the costandard modules of $S\otimes_R A$ are of the form $S\otimes_R \Cs(\l)$. Since the exact sequences given by filtrations are all $(A, R)$-exact,  the functor $S\otimes_R -$ is exact on the exact sequences of Proposition \ref{tiltingconstruction}. 
		Therefore, $S\otimes_R T$ is a characteristic tilting module for $S\otimes_R A$.
	\end{proof}
	
	\begin{Remark}
		We cannot expect that the isomorphism $T(\l)(\mi)\simeq T_{(\mi)}(\l)$ holds in this generality, where $T_{(\mi)}(\l)$ is a partial tilting indecomposable module of $A(\mi)$ for $\mi$ a maximal ideal of $R$. For example, the rank of $(U_i)$ at each localization $\mi$ ($\mi$ a maximal ideal of $R$) might not be constant for some $i$.
	\end{Remark}

An analogue of the following result in a slightly different setup can be found in \citep[Proposition 2.11]{zbMATH06751586}.
	
	\begin{Prop}\label{tensorproductstandcostand}
		Let $(A, \{\Delta(\lambda)_{\lambda\in \Lambda}\})$ be a split quasi-hereditary algebra. Suppose that $M\in \mathcal{F}(\Stsim)$ and $N\in \mathcal{F}(\Cssim)$. 
		Then, the following assertions hold.
		\begin{enumerate}[(a)]
			\item The functor $-\otimes_A M\colon \mathcal{F}(D\Cssim)\rightarrow R\proj$ is a well-defined exact functor.
			\item The functor $DN\otimes_A -\colon \mathcal{F}(\Stsim)\rightarrow R\proj$ is a well-defined exact functor.
		\end{enumerate}
	\end{Prop}
	\begin{proof}
		It is enough to show that $DN\otimes_A M\in R\proj$ and $\Tor_{i>0}^A(DN, M)=0$.  Let $\mi$ be a  maximal ideal of $R$.  Let $M^{\bullet}$ be a deleted projective (left) $A$-resolution of $M$. Since $M\in R\proj$, $M^{\bullet}(\mi)=R(\mi)\otimes_R M^{\bullet}$ is a deleted projective $A(\mi)$-resolution of $M(\mi)$ for every maximal ideal $\mi$ in $R$.  Further, each module in the complex $DN\otimes_A M^{\bullet}$ belongs to $\add_R DN$. So, the complex $DN\otimes_A M^{\bullet}$ is a flat chain complex. Consider the K\"unneth spectral sequence for chain complexes (see for example \citep[Theorem 5.6.4]{Weibel2003})
		\begin{align*}
			E^2_{p, q}= \Tor_p^R(\Tor_q^A(DN, M), R(\mi))\implies H_{p+q}(DN\otimes_A M^{\bullet}\otimes_R R(\mi))=\Tor_{p+q}^{A(\mi)}(DN(\mi), M(\mi)).
		\end{align*}
		Observe that \begin{align}
			\Tor_{i>0}^{A(\mi)}(DN(\mi), M(\mi))&=\Tor_{i>0}^{A(\mi)}(D_{(\mi)}N(\mi), M(\mi))=H_{i>0}(D_{(\mi)} N(\mi)\otimes_{A(\mi)}M^{\bullet}(\mi)  )\\
			&\simeq H_{i>0}(D_{(\mi)}\Hom_{A(\mi)}(M^{\bullet}(\mi), N(\mi)))\simeq D_{(\mi)}H^{i>0}(\Hom_{A(\mi)}(M^{\bullet}(\mi), N(\mi))) \nonumber\\
			&\simeq D_{(\mi)} \Ext_{A(\mi)}^{i>0}(M(\mi), N(\mi))=0.
		\end{align}The last equality follows from Proposition \ref{standardscotiltingsreductiontofields} and Lemma \ref{extvanishesforfiltratedmodules}.
		
		By \citep[Lemma A.3]{CRUZ2022410}, for each maximal ideal $\mi$ in $R$, we obtain that \begin{align}
			0=E_{1, 0}^2=\Tor_1^R(DN\otimes_A M, R(\mi)).
		\end{align}Therefore, $DN\otimes_A M\in R\proj$. Moreover, $E^2_{i, 0}=0$ for all $i>0$. Again, by \citep[Lemma A.3]{CRUZ2022410}, it follows that \begin{align}
			\Tor_1^A(DN, M)(\mi) = E^2_{0, 1}\simeq E^2_{2, 0}=0.
		\end{align}Thus, $\Tor_1^A(DN, M)=0$ and consequently $E^2_{i, 1}=0$ for all $i\geq 0$. We can proceed by induction on $q$ to show that $E^2_{i, j}=0$ for all $i\geq 0$, $1\leq j\leq q$. In fact, assume that $E^2_{i, j}=0$ for all $i\geq 0$, $1\leq j\leq q$ for a given $q$. By \citep[Lemma A.4]{CRUZ2022410}, there exists an exact sequence
		\begin{align}
			0=E^2_{q+2, 0}\rightarrow E^2_{0, q+1}\rightarrow H_{q+1}=0.
		\end{align}So, $\Tor_{q+1}(DN, M)(\mi)=0$. Hence, $\Tor_{q+1}(DN, M)=0$. Therefore, $E^2_{i, q+1}=0$ for all $i\geq 0$. We showed that $E^2_{i, j}=0$ for all $i\geq 0$ and $j\geq 1$. This means that $\Tor_{q>0}^A(DN, M)=0$.
	\end{proof}
	
	As a direct consequence of Proposition \ref{tensorproductstandcostand}.
	\begin{Cor}\label{homofstandardcostandardisproj} Let $(A, \{\Delta(\lambda)_{\lambda\in \Lambda}\})$ be a split quasi-hereditary algebra.
		Let $M\in \mathcal{F}(\Stsim)$ and let ${N\in \mathcal{F}(\Cssim)}$. Then, $\Hom_A(M, N)\in R\proj$.
	\end{Cor}We can say even more about $\Hom_A(M, N)$.
	In fact, applying the same idea used in \citep{Koenig2001} to construct a filtration to $\End_A(T)$, for $T$ we can construct a filtration for $\Hom_A(M, N)$.
	\begin{Prop}\label{filtrationofhom}
		Let $(A, \{\Delta(\lambda)_{\lambda\in \Lambda}\})$ be a split quasi-hereditary algebra. Let $M\in \mathcal{F}(\Stsim),$  $L\in\mathcal{F}(\Cssim)$. Let $\L\rightarrow \{1, \ldots, t\}$, $\l\mapsto i_\l$ be an increasing bijection and set $\St_{i_\l}:=\St(\l)$. So, there exists $U_i, S_i\in R\proj$ such that\begin{align*}
			&0=M_{n+1}\subset M_n\subset \cdots \subset M_1=M \text{ with } M_i/M_{i+1}\simeq \St_i\otimes_R U_i\\
			&0=L_{n+1}\subset L_n\subset \cdots \subset L_1=L \text{ with } L_i/L_{i+1}\simeq \Cs_{n-i+1}\otimes_R S_{n-i+1}, \ i=1, \ldots, n.
		\end{align*} Then, $\Hom_A(M, L)$ has a filtration \begin{equation*}
			\begin{gathered}
				0=X_{n+1}\subset X_n\subset X_{n-1}\subset \cdots\subset X_1=X=\Hom_A(M, N), \\ X_i=\Hom_A(M/M_{n-i+2}, L_i)=\Hom_{A/J_{n-i+2}}(M/M_{n-i+2}, L_i), \quad X_i/X_{i+1}\simeq \Hom_R(U_{n-i+1}, S_{n-i+1}).
			\end{gathered}	
		\end{equation*}
	\end{Prop}
	\begin{proof}
		We will proceed by induction on $n=|\L|$. Assume $n=1$. Then, $M\simeq \St_1\otimes_R U_1$ and $L\simeq \Cs_1\otimes_R S_1$. Then, \begin{align}
			\Hom_A(M, N)&=\Hom_A(\St_1\otimes_R U_1, \Cs_1\otimes_R S_1)\simeq \Hom_R(U_1, \Hom_A(\St_1, \Cs_1\otimes_R S_1))\\&\simeq \Hom_R(U_1, \Hom_A(\St_1, \Cs_1)\otimes_R S_1)\simeq \Hom_R(U_1, S_1).
		\end{align} So, the filtration $0\subset \Hom_R(U_1, S_1)=X_1$ is the desired one. Assume the result holds for $n-1$. Consider the short exact sequences
		\begin{align}
			0\rightarrow \St_n\otimes_R U_n\xrightarrow{k_M} M\xrightarrow{\pi_M}M/M_n\rightarrow 0\label{eq152}\\
			0\rightarrow L_2\xrightarrow{k_L}L\xrightarrow{\pi_L}\Cs_n\otimes_R S_n\rightarrow 0.\label{eq153}
		\end{align}
		Applying the functor $\Hom_A(M, -)$ to (\ref{eq153}) gives 
		\begin{align*}
			0\rightarrow \Hom_A(M, L_2)\xrightarrow{\Hom_A(M, k_L)} \Hom_A(M, L)\xrightarrow{\Hom_A(M, \pi_L)} \Hom_A(M, \Cs_n\otimes_R S_n)\rightarrow \Ext_A^1(M, L_2)=0. 
		\end{align*}Applying $\Hom_A(-, L)$ to (\ref{eq152}) gives
		\begin{align*}
			\Hom_A(M/M_n, L)\xhookrightarrow{\Hom_A(\pi_M, L)} \Hom_A(M, L)\xrightarrow{\Hom_A(k_M, L)} \Hom_A(\St_n\otimes_R U_n, L) \rightarrow \Ext_A^1(M/M_n, L)=0.
		\end{align*}
		Applying the functor $\Hom_A(-, L_2)$ to (\ref{eq152}) we get the following exact sequence
		\begin{equation*}
			\begin{tikzcd}
				0 \arrow[r]   &[-1.5em] \Hom_A(M/M_n, L_2) \arrow[r, "\Hom_A(\pi_M{,} L_2)", outer sep=0.75ex] & \Hom_A(M, L_2) \arrow[r, "\Hom_A(k_M{,} L_2)", outer sep=0.75ex] & \Hom_A(\St_n\otimes_R U_n, L_2)
				\arrow[r]   &[-1em] 0=\Ext_A^1(M/M_n, L_2). 
			\end{tikzcd}\label{eq156}
		\end{equation*}
		Since $L_2\in \mathcal{F}(\Cs_{i<n})$ we obtain $\Hom_A(\St_n\otimes_R U_n, L_2)=0$. Hence, $\Hom_A(\pi_M, L_2)$ is an isomorphism. Applying the functor $\Hom_A(-, \Cs_n\otimes_R S_n)$ to (\ref{eq152}) yields the exact sequence
		\begin{equation*}
			\begin{tikzcd}
				& \Hom_A(M/M_n, \St_n\otimes_R S_n) \arrow[r, hookrightarrow] & \Hom_A(M, \St_n\otimes_R U_n) \ar[r, "\Hom_A(k_M{,} \St_n\otimes_R U_n)", outer sep=0.75ex, twoheadrightarrow] & \Hom_A(\St_n\otimes_R U_n, \Cs_n\otimes_R S_n)  
			\end{tikzcd}.
		\end{equation*}
		Since $M/M_n\in \mathcal{F}(\Stsim_{i<n})$ we obtain $\Hom_A(M/M_n, \Cs_n\otimes_R S_n)=0$. Hence, $\Hom_A(k_M, \St_n\otimes_R U_n)$ is an isomorphism.
		Therefore, we have an exact sequence
		\begin{equation*}
			\begin{tikzcd}[sep=7em]
				\Hom_A(M/M_n, L_2) \arrow[ hookrightarrow, r, "\Hom_A(M{,} k_L)\circ \Hom_A(\pi_M{,} L_2)", outer sep=0.75ex] & \Hom_A(M, L) 
				\arrow[r, "\Hom_A(k_M{,} \St_n\otimes_R U_n)\circ \Hom_A(M{,} \pi_L)", twoheadrightarrow, outer sep=0.75ex]  
				& \Hom_A(\St_n\otimes_R U_n, \Cs_n\otimes_R S_n).  
			\end{tikzcd}
		\end{equation*}
		Furthermore, \begin{align}
			\Hom_A(\St_n\otimes_R U_n, \Cs_n\otimes_R S_n)&\simeq \Hom_R(U_n, \Hom_A(\St_n, \Cs_n\otimes_R S_n))\simeq \Hom_R(U_n, \Hom_A(\St_n, \Cs_n)\otimes_R S_n)\nonumber \\&\simeq \Hom_R(U_n, S_n).
		\end{align} Fix $J_n=\im \tau_{\St_n}$. Because of $M/M_n\in \mathcal{F}(\Stsim_{i<n})$ and $L_2\in \mathcal{F}(\Cssim_{i<n})$, we have $\Hom_A(M/M_n, L_2)=\Hom_{A/J_n}(M/M_n, L_2)$. Therefore, $X/\Hom_{A/J_n}(M/M_n, L_2)\simeq \Hom_R(U_n, S_n)$. By induction, \linebreak$\Hom_{A/J_n}(M/M_n, L_2)$ admits a filtration
		$
		0\subset X_n\subset X_{n-1}\subset \cdots \subset X_2=\Hom_A(M/M_n, L_2),
		$ with $X_i\simeq \Hom_{A/J_n/J_{n-i+2}/J_n}(M/M_{n-i+2}, L_i)\simeq \Hom_{A/J_{n-i+2}}(M/M_{n-i+2}, L_i), \ i=2, \ldots n$. Thus,  
		$
		0\subset X_n\subset X_{n-1}\subset \cdots\subset X_2\subset X 
		$ is the desired filtration.
	\end{proof}
	
	The following result has been observed in the literature several times in particular cases (see for example Lemma 4.2 of \citep{zbMATH01294605}). 
	\begin{Cor}\label{hommaximalforfiltrations}
		Let $(A, \{\Delta(\lambda)_{\lambda\in \Lambda}\})$ be a split quasi-hereditary algebra.	Let $M\in \mathcal{F}(\Stsim)$ and let $N\in \mathcal{F}(\Cssim)$. Let $Q$ be a commutative $R$-algebra and commutative Noetherian ring. Then, $$Q\otimes_R \Hom_A(M, N)\simeq \Hom_{Q\otimes_RA}(Q\otimes_R M, Q\otimes_R N).$$
	\end{Cor}
	\begin{proof}It follows immediately from Proposition \ref{tensorproductstandcostand}. However, the following argument using Proposition \ref{filtrationofhom} is better equipped for some relevant cases.
		We shall proceed by induction on $n=|\L|$. Assume $n=1$. Then, $\Hom_A(M, N)\simeq \Hom_R(U_1, S_1)$. So, \begin{align}
			Q\otimes_R\Hom_A(M, N)&\simeq 	Q\otimes_R \Hom_R(U_1, S_1)\simeq \Hom_{Q\otimes_R R}(Q\otimes_RU_1, Q\otimes_R S_1)\\&\simeq \Hom_{Q\otimes_RA}(Q\otimes_R\St_1\otimes_{Q\otimes_R R}Q\otimes_RU_1, Q\otimes_R\Cs_1\otimes_{Q\otimes_R R}Q\otimes_R S_1)\\&\simeq \Hom_{Q\otimes_RA}(Q\otimes_RM, Q\otimes_RN).
		\end{align} Assume that the result holds for $n-1$.
		Consider $A$ with $|\L|=n$. Consider the exact sequence given by the filtration of $\Hom_A(M, N)$:
		\begin{align}
			0\rightarrow \Hom_{A/J}(M/M_n, L_2)\rightarrow \Hom_A(M, L)\rightarrow \Hom_R(U_n, S_n )\rightarrow 0.\label{eq165}
		\end{align}Since $\Hom_R(U_n, S_n)\in R\proj$, (\ref{eq165}) is $(A, R)$-exact. We will denote by $X(Q)$ the tensor product $Q\otimes_R X$. Applying $Q\otimes_R -$ we get the following commutative diagram with exact rows
		\begin{center}
			\hspace*{-1em}	\begin{tikzpicture}[baseline= (a).base]
				\node[scale=.85] (a) at (0,0){	\begin{tikzcd}
						Q\otimes_R\Hom_{A/J}(M/M_n, L_2)\arrow[r, hookrightarrow]\arrow[d, "\alpha_1"]&[-0.75em] Q\otimes_R\Hom_A(M, L)\arrow[r, twoheadrightarrow]\arrow[d, "\alpha"]&[-0.75em]Q\otimes_R \Hom_A(\St_n\otimes_R U_n, \Cs_n\otimes_R S_n)\arrow[d, "\alpha_2"] \\
						\Hom_{ A/J(Q)}( M/M_n(Q),  L_2(Q))\arrow[r, hookrightarrow] &[-0.75em]\Hom_{A(Q)}(M(Q), L_2(Q))\arrow[r, twoheadrightarrow]&[-0.75em] \Hom_{A(Q)}(\St_n(Q)\otimes_{Q} U_n(Q), \Cs_n(Q)\otimes_Q S_n(Q))
					\end{tikzcd}.};
			\end{tikzpicture}
		\end{center}
		Note that the bottom row is exact since we use the same exact sequences given by filtrations of \linebreak\mbox{$M(Q)\in \mathcal{F}(\St(Q))$} and $L(Q)\in \mathcal{F}(\Cs(Q))$ in view of Proposition 4.14 of \citep{Rouquier2008} (see also \citep[Proposition 3.1.1]{cruz2021cellular}). This is admissible because all the modules involved in the filtrations are projective over $R$. So, the functor $Q\otimes_R -$ preserves the given filtrations.
		By induction, $\alpha_1$ is an isomorphism. Since \mbox{$\St_n\otimes_R U_n\in A\proj$,} $\alpha_2$ is an isomorphism. By Snake Lemma, $\alpha$ is an isomorphism.
	\end{proof}

	\begin{Prop}\label{standardscotiltingsreductiontofields}
		\citep[Proposition 4.30]{Rouquier2008} Let $(A, \{\Delta(\lambda)_{\lambda\in \Lambda}\})$ be a split quasi-hereditary algebra. Let \mbox{$M\in A\m$.} Then, the following assertions hold.\begin{enumerate}[(a)]
			\item $M\in \mathcal{F}(\Stsim)$ if and only if $M(\mi)\in \mathcal{F}(\St(\mi))$ for all maximal ideals $\mi$ of $R$ and $M\in R\proj$.
			\item $M\in \mathcal{F}(\Cssim)$ if and only if $M(\mi)\in \mathcal{F}(\Cs(\mi))$ for all maximal ideals $\mi$ of $R$ and $M\in R\proj$.
			\item Let $T$ be a characteristic tilting module. $M\in \add T$ if and only if $M(\mi)\in \add T(\mi)$ for all maximal ideals $\mi$ of $R$ and $M\in R\proj$.
		\end{enumerate}  
	\end{Prop}
	\begin{proof}
		Assume that $M\in \mathcal{F}(\Stsim)$. The functor $R(\mi)\otimes_R -\colon A\m\rightarrow A(\mi)\m$ is exact on filtrations of $M$ and therefore $M(\mi)\in \mathcal{F}(\St(\mi))$.
		
		Conversely, let $M\in A\m\cap R\proj$ such that $M(\mi)\in \mathcal{F}(\St(\mi))$ for every maximal $\mi$ in $R$. Let $N\in \mathcal{F}(\Cssim)$. In particular, $N(\mi)\in \mathcal{F}(\Cs(\mi))$. The projectivity of $M$ as $R$-module implies that $R(\mi)\otimes_R -$ sends a deleted projective $A$-resolution of $M$ to a deleted projective $A(\mi)$-resolution of $M(\mi)$, $\mi\in \MaxSpec R$. Using the  K\"unneth spectral sequence for chain complexes and the argument developed in the proof of Proposition \ref{tensorproductstandcostand} we obtain that $\Tor_{i>0}^A(DN, M)=0$ and $DN\otimes_A M\in R\proj$. This implies that $\Ext_A^{i>0}(M, N)=0$. In fact, $DN\otimes_A M^{\bullet}$ is an exact sequence, where $M^{\bullet}$ denotes a projective resolution of $M$. Since $DN\otimes_A M\in R\proj$ applying $D$ we obtain that $\Hom_A(M^{\bullet}, DDN)$ is an exact sequence. By Theorem \ref{filtrationsintermsofext}, (a) follows.

		Let $N\in \mathcal{F}(\Cssim)$. Then, again since the functor $R(\mi)\otimes_R -\colon A\m\rightarrow A(\mi)\m$ is exact on filtrations of $N$, $N(\mi)\in \mathcal{F}(\Cs(\mi))$, $\mi\in \MaxSpec R$.
		Conversely, assume that $N\in R\proj$ and $N(\mi)\in \mathcal{F}(\Cs(\mi))$ for every maximal ideal $\mi$ in $R$. Then, $DN(\mi)\simeq\Hom_{R(\mi)}(N(\mi), R(\mi))\in \mathcal{F}(\St_{A^{op}}(\mi))$. By (a), $DN\in \mathcal{F}(\Stsim_{A^{op}})$ hence $N\in \mathcal{F}(\Cssim)$. So, (b) follows.
		
		Applying (a) and (b) to Theorem \ref{severalpropertiestilting}, (c) follows. 
	\end{proof}
	
	\begin{Prop}\label{standardsrreductionfields}
		Let $(A, \{\Delta(\lambda)_{\lambda\in \Lambda}\})$ be a split quasi-hereditary algebra. Let $M\in A\m$. Then, the following assertions hold.
		\begin{enumerate}
			\item If $M\in \mathcal{F}(\Stsim)$ and $M(\mi)\simeq \St(\l)(\mi)$ for some $\l\in \L$ for every maximal ideal $\mi$ of $R$, then $M\simeq \St(\l)\otimes F$ for some $F\in Pic(R)$.
			\item If $M\in \mathcal{F}(\Cssim)$ and $M(m)\simeq \Cs(\l)(\mi)$ for some $\l\in \L$ for every maximal ideal $\mi$ of $R$, then $M\simeq \Cs(\l)\otimes F$ for some $F\in Pic(R)$.
		\end{enumerate}
	\end{Prop}
	\begin{proof}
		Let $\L\rightarrow \{1, \ldots, t\}$, $\l\mapsto i_\l$ be an increasing bijection and set $\St_{i_\l}:=\St(\l)$.	Since $M\in \mathcal{F}(\Stsim)$ there is a filtration \begin{align}
			0=M_{n+1}\subset M_n\subset \cdots \subset M_1=M, \quad M_i/M_{i+1}\simeq \St_i\otimes_R U_i, \ U_i\in R\proj.
		\end{align}
		By Proposition \ref{filtrationofhom}, $
		\Hom_A(M, \Cs_i) \simeq \Hom_A(M/M_{i+1}, \Cs_i)\simeq \Hom_R(U_i, R)=DU_i.
		$
		Let $\mu\in \L\setminus \{\l \}$. Then, $
		\Hom_A(M, \Cs({\mu}))(\mi)\simeq \Hom_{A(\mi)}(M(\mi), \Cs({\mu})(\mi))\simeq \Hom_{A(\mi)}(\St(\l)(\mi), \Cs(\mu)(\mi))=0
		$ for every maximal ideal $\mi$ in $R$. Hence,  $DU(\mu)\simeq\Hom_A(M, \Cs(\mu))=0$, and thus ${M\simeq \St(\l)\otimes_R U(\l)}$ since each $U(\mu)\in R\proj$. Observe that $\St(\l)\simeq M(\mi)\simeq \St(\l)(\mi)\otimes_{R(\mi)} U(\l)(\mi)$ for every $\mi\in \MaxSpec R$. Hence, $U(\l)(\mi)\simeq R(\mi)$. Since $U(\l)\in R\proj$, Remark \ref{invertiblemodulesresidue} yields that $U(\l)\in Pic(R)$.
	\end{proof} Using the base change property presented in Corollary \ref{homofstandardcostandardisproj}, we can improve Proposition \ref{costandardsundergroundringchange} by seeing that the invertible module must be the regular module itself.

		\begin{Prop}\label{costandardsundergroundringchangetwo}
		Let $S$ be a commutative $R$-algebra and a Noetherian ring. 	Let $(A, \{\Delta(\lambda)_{\lambda\in \Lambda}\})$ be a split quasi-hereditary algebra. Then, the following assertions hold.
		\begin{enumerate}[(a)]
			\item The costandard modules of  $(S\otimes_R A, \{S\otimes_R \Delta(\lambda)_{\lambda\in \Lambda}\})$ are of the form $S\otimes_R \Cs(\l)$ with $\l\in \L$.
			\item Let $T=\bigoplus_{\l\in \L} T(\l)$ be a characteristic tilting module of $A$. Then $S\otimes_R T(\l)$ satisfies (\ref{eq121}) and (\ref{eq121b})) over $S\otimes_RA$ and $S\otimes_R T$ is a characteristic tilting module.
		\end{enumerate} 
	\end{Prop}
\begin{proof}
	By Proposition \ref{costandardsundergroundringchange}, for each $\l\in \L$ we can obtain $S\simeq \Hom_{S\otimes_R A}(S\otimes_R \St(\l), S\otimes_R \Cs(\l)\otimes_R U(\l))\simeq \Hom_{S\otimes_R A}(S\otimes_R \St(\l), S\otimes_R \Cs(\l))\otimes_SU_\l$. By Corollary \ref{homofstandardcostandardisproj}, the latter is isomorphic to \begin{align*}
		S\otimes_R \Hom_A(\St(\l), \Cs(\l))\otimes_S U(\l)\simeq (S\otimes_R R)\otimes_S U(\l)\simeq S\otimes_S U(\l)\simeq U(\l).
	\end{align*}
So, (a) follows. Using (a) and the same argument as in Proposition \ref{costandardsundergroundringchange}, (b) is clear.
\end{proof}

	\section{$\mathcal{F}(\Stsim)$ determines all equivalences of split highest weight categories} \label{Structure of split quasi-hereditary algebras}
	
	The full subcategory of $A\m $ $\mathcal{F}(\Stsim)$ completely determines the split quasi-hereditary algebra $A$.
	This result over finite-dimensional algebras appears in \cite{zbMATH00140218} and in \cite{MR1284468}.

	\begin{Theorem}\label{standardsgivesthewholealgebra}
		Let $(A\m, \{\Delta(\lambda)_{\lambda\in \Lambda}\})$ and $(B\m, \{\Omega(\theta)_{\theta\in \Theta}\})$ be two split highest weight categories. There exists an exact equivalence between $\mathcal{F}(\Stsim)$ and $\mathcal{F}(\tilde{\Omega})$ if and only if $A\m$ and $B\m$ are equivalent as split highest weight categories in the sense of \cite{Rouquier2008}.
	\end{Theorem}
	\begin{proof}
		Assume that $A\m$ and $B\m$ are equivalent as split highest weight categories. By definition, there exists an exact equivalence functor $F\colon A\m\rightarrow B\m$ satisfying $F\St(\l)\simeq \Omega(\phi(\l))\otimes_R U_\l$, $\l\in \L$, $U_\l\in Pic(R)$. It follows that the restriction of $F$ to $\mathcal{F}(\Stsim)$ has image in $\mathcal{F}(\tilde{\Omega})$ which is again fully faithful and exact.
		
		Conversely, let $H\colon \mathcal{F}(\Stsim)\rightarrow \mathcal{F}(\tilde{\Omega})$ and $G\colon \mathcal{F}(\tilde{\Omega})\rightarrow \mathcal{F}(\Stsim)$ be exact equivalences. We claim that $H$ sends projective $A$-modules to projective $B$-modules. Let $P\in A\proj$. Let $0\rightarrow \Omega_\theta \rightarrow X\rightarrow HP\rightarrow 0$ be a short exact sequence  for some $\theta\in \Theta$ and denote it by $\rho$. By assumption, $HP\in \mathcal{F}(\tilde{\Omega})$ and so $\mathcal{F}(\tilde{\Omega})$ being closed under extensions implies that $\rho$ is an exact sequence of modules belonging to $\mathcal{F}(\tilde{\Omega})$.  Since $GHP\simeq P$ applying $G$ to $\rho$  yields a split exact sequence over $A$. Hence the exact sequence $HG\rho$ is also split and being equivalent to $\rho$ we obtain that $\rho$ splits over $B$. This means that $\Ext_B^1(HP, \Omega(\theta))=0$ for all $\theta\in \Theta$. 
		By Theorem \ref{projectiveandinjectiveinsidefiltrations}, $HP\in B\proj$. In particular, $HA\in B\proj$. Symmetrically, $GQ\in A\proj$ for all $Q\in B\proj$, and in particular, $GB\in A\proj$. Therefore $GB\bigoplus K\simeq A^s$ for some $K\in A\m$ and some $s\in \mathbb{N}$. Applying $H$ yields that $B\bigoplus HK\simeq HA^s$. Therefore, $HA$ is a $B$-progenerator. 	By Morita theory, the functor $\Hom_B(HA, -)\colon B\m\rightarrow A\m$ is an exact equivalence of categories. In particular, this functor lifts $G$. In fact, for any $X\in \mathcal{F}(\tilde{\Omega})$, $\Hom_B(HA, X)\simeq \Hom_A(GHA, GX)\simeq \Hom_A(A, GX)\simeq GX$. By applying $R(\mi)\otimes_R -$ it follows that the cardinality of $\L$ and of $\Theta$ coincide with the number of non-isomorphic simple $A(\mi)$-modules.
		
		Let $\L\rightarrow \{1, \ldots, |\L|\}$, $\l\mapsto i_\l$ be an increasing bijection and set $\St_{i_\l}:=\St(\l)$.	We now claim that for each $t=1, \ldots, |\L|$ there exists $\theta_t\in \Theta$ and $U_t\in R\proj$ so that $H\St_t\simeq \Omega(\theta_t)\otimes_R U_t$.
		We shall proceed by reverse induction on $t$. 	Assume the result is known for some $t\in \mathbb{N}$ with $1<t<|\L|$. 
		First we must observe that $H\St_{t-1}\in \mathcal{F}(\tilde{\Omega}_{\Theta\setminus\{\theta_{{|\L|}}, \ldots, \theta_{{t}} \} })$. In fact, assume that $l$ is the highest element in $\{t, \ldots, |\L|\}$ so that $\Omega(\theta_l)$ appears in a filtration of $H\St_{t-1}$. Then there exists an exact sequence $0\rightarrow \Omega(\theta_l)\otimes_R S_l\rightarrow H\St_{t-1}\rightarrow X\rightarrow 0$, with $S_l\in R\proj$. Applying $G$ after $U_l\otimes_R -$ yields by construction that there exists a non-zero homomorphism from $\St_l\otimes_R S_l$ to $\St_{t-1}\otimes_R U_l$ which cannot happen because $l>t-1$ (the latter can be seen by applying $R(\mi)\otimes_R -$ to this $(A, R)$-monomorphism).
		Thus, there exists an exact sequence $0\rightarrow \Omega(\nu)\otimes_R S_\nu \rightarrow H\St_{t-1}\xrightarrow{\pi} X\rightarrow 0$ with $\nu\in \Theta\setminus\{\theta_{{|\L|}}, \ldots, \theta_{{t+1}} \}$ a maximal element, $X\in \mathcal{F}(\tilde{\Omega}_{\Theta\setminus\{\theta_{{|\L|}}, \ldots, \theta_{{t+1}}, \nu \} })$ and $S_\nu\in R\proj$. If $\pi=0$ then we fix $\theta_{t-1}=\nu$. Otherwise, applying $G$ yields an exact sequence $0\rightarrow G \Omega(\nu)\otimes_R S_\nu\rightarrow \St_{t-1}\rightarrow GX\rightarrow 0$ with $GX\in \mathcal{F}(\Stsim_{l<t})$. It follows that since $G\pi\neq 0$ that $GX\simeq \St_{t-1}\otimes_R S_{t-1}$ for some $S_{t-1}\in R\proj$. Further since $R(\mi)\otimes_R -$ is right exact the surjective map $\St_{t-1}\rightarrow \St_{t-1}\otimes_R S_{t-1}$ induced by $G\pi$ becomes an isomorphism over $A(\mi)$ for  all $\mi\in \MaxSpec R$. By Nakayama's Lemma this map must be an isomorphism over $A$ and so $G\pi$ is also an isomorphism. This means that $S_\nu=0$ and $H\St_{t-1}\in\mathcal{F}(\tilde{\Omega}_{\Theta\setminus\{\theta_{{|\L|}}, \ldots, \theta_{{t+1}}, \nu \} })$.  By going through all possible standard modules that might appear in the filtration of $H\St_{t-1}$ we obtain (eventually after a finite number of steps)  that $H\St_{t-1} \simeq \Omega(\theta_{t-1})\otimes_R U_{t-1}$ for some $U_{t-1}\in R\proj$,  and $\theta_{t-1}\in \Theta\setminus\{\theta_{{|\L|}}, \ldots, \theta_{{t+1}} \} $.
		Using the same reasoning, we obtain that the result holds for $\St_{|\L|}\in A\proj$ and so the claim follows.
		
		Moreover, we constructed an increasing bijection $\{1, \ldots, |\L| \}\rightarrow \Theta$ and so a bijection of posets ${\phi\colon \L\rightarrow \Theta}$.
		Symmetrically applying the same arguments reversing the roles of $\St$ and $\Omega$, we obtain that, for all $t\in \{1, \ldots, |\L| \}$, there exists $i_t\in \{1, \ldots, |\L|\}$ so that $G\Omega(\theta_t)\simeq \St_{i_t}\otimes_R F_t$ for some $F_t\in R\proj$. Hence, $\Omega(\theta_t)\simeq H\St_{i_t}\otimes_R F_t\simeq \Omega(\theta_{i_t})\otimes_R U_{i_t}\otimes_R F_t$. It follows that $\theta_t=\theta_{i_t}$ and both $F_t$ and $U_{i_t}$ are invertible $R$-modules.
		
		It follows that $\Hom_B(HA, \Omega(\phi(\mu)))\simeq G\Omega(\phi(\mu))\simeq \St(\mu)\otimes_R U_\mu$, $U_\mu\in Pic(R)$, for all $\mu\in \L$.
	\end{proof}

	\section{Ringel duality}\label{Ringel duality}
	
	Characteristic tilting modules allow us also in the integral setup to construct new split quasi-hereditary algebras, and to see that in fact split quasi-hereditary algebras over Noetherian rings do come in pairs. The origin of this phenomenon is in \cite{MR1128706} inspired by the work developed in \cite{zbMATH00010169}. The former studies the endomorphism algebra of a characteristic tilting module over an Artinian quasi-hereditary algebra.

	\begin{Lemma}\label{functorringeldualpart1}
		Let $(A, \{\Delta(\lambda)_{\lambda\in \Lambda}\})$ be a split quasi-hereditary algebra. Assume that $T=\sumSt T(\l)$ is a characteristic tilting module. Fix $B=\End_A(T)^{op}$. 
		Then the following assertions hold.
		\begin{enumerate}[(a)]
			\item The functor $G=\Hom_A(T, -)\colon A\M\rightarrow B\M$ restricts to an exact equivalence between $\mathcal{F}(\Cssim)$ and $\mathcal{F}(\Stsim_B)$ with $\St_B(\l):=G\Cs(\l)$, $\l\in\L$. 
			\item $(B, \{\Delta_B(\lambda)_{\lambda\in \Lambda^{op}}\})$ is a split quasi-hereditary $R$-algebra, where $\L^{op}$ is the set $\L$ together with the partial order $\leq_B$ defined in the following way:
			$\l\leq_B \mu $ if and only if $\l\geq \mu$.
		\end{enumerate}
	\end{Lemma}
	\begin{proof}
		The functor $\Hom_A(T, -)$ is exact on $\mathcal{F}(\Cssim)$ because $\Ext_A^1(T, M)=0$ for every $M\in \mathcal{F}(\Cssim)$.
		Let $N\in \mathcal{F}(\Cssim)$. Let $\St\rightarrow \{1, \ldots, n\}$, $\St_i\mapsto i$ be an increasing bijection.  Hence, we have a filtration \begin{align}
			0\subset I_1\subset \cdots \subset I_n=N, \quad I_i/I_{i-1}\simeq I_i\otimes_R U_i, \ i=1, \ldots, n.
		\end{align} Applying $\Hom_A(T, -)$ yields the exact sequence
		\begin{align}
			0\rightarrow \Hom_A(T, I_{i-1})\rightarrow \Hom_A(T, I_i)\rightarrow \Hom_A(T, \Cs_i\otimes_R U_i)\rightarrow 0.
		\end{align} By Lemma \ref{tensorprojcommutingonHom}, $\Hom_A(T, \Cs_i\otimes_R U_i)\simeq \Hom_A(T, \Cs_i)\otimes_R U_i$.
		So, $\Hom_A(T, -)$ sends a module $N\in \mathcal{F}_A(\Cssim)$ to $\Hom_A(T, N)\in \mathcal{F}_B(\widetilde{\Hom_A(T, \Cs))}$. Fix $\St_B(i)=G\Cs_i$. We shall now prove that $G$ is full and faithful on $\mathcal{F}(\Cssim)$. Let $Y\in A\m$. Then,
		\begin{align}
			\Hom_A(T, Y)\simeq G(Y)=\Hom_B(B, GY)\simeq \Hom_B(\Hom_A(T, T), GY)\simeq \Hom_B(GT, GY).
		\end{align}Hence, for any $X\in \add T$, we have $
		\Hom_A(X, Y)\simeq \Hom_B(GX, GY)$ for all  $Y\in A\m.
		$ Let $X\in \mathcal{F}_A(\Cssim)$. By Theorem \ref{severalpropertiestilting}, there exists an $\add T$-presentation
		$
		T_1\rightarrow T_0\rightarrow X\rightarrow 0
		$. Applying $\Hom_A(-, Y)$ and $\Hom_B(G-, GY)$ we obtain the following commutative diagram with exact rows
		\begin{center}
			\begin{tikzcd}
				0\arrow[r]& \Hom_A(X, Y) \arrow[r]\arrow[d]& \Hom_A(T_0, Y)\arrow[r]\arrow[d, "\simeq"]& \Hom_A(T_1, Y)\arrow[d, "\simeq"]\\
				0\arrow[r]& \Hom_B(GX, GY)\arrow[r]&\Hom_B(GT_0, GY)\arrow[r]&\Hom_B(GT_1, GY)
			\end{tikzcd}.
		\end{center} 
		By diagram chasing, $\Hom_A(X, Y)\simeq \Hom_B(GX, GY)$ for all $X, Y\in \mathcal{F}_A(\Cssim)$. 
		
		Now we claim that $\Ext_A^1(U_i\otimes_R \Cs_i, N)\simeq \Ext_B^1(G(U_i\otimes_R \Cs_i), GN)$ for all $N\in \mathcal{F}_A(\Cssim)$ and $U_i\in R\proj$.
		
		Consider the exact sequence
		$
		0\rightarrow Y_i\rightarrow T_i\rightarrow \Cs_i\rightarrow 0.
		$ Applying $U_i\otimes_R -$ we get the exact sequence
		\begin{align}
			0\rightarrow U_i\otimes_R Y_i\rightarrow U_i\otimes_R T_i\rightarrow U_i\otimes_R \Cs_i\rightarrow 0.
		\end{align}Let $N\in \mathcal{F}_A(\Cssim)$. Recall that $G(T_i\otimes_R U_i)\in B\proj$, so applying $\Hom_A(-, N)$ and $\Hom_B(-, GN)\circ G$ we obtain the following commutative diagram with exact rows
		\begin{center}
			\begin{tikzcd}
				\Hom_A(T_i\otimes_R U_i, N)\arrow[r]\arrow[d, "\simeq"]& \Hom_A(Y_i\otimes_R U_i, N)\arrow[r]\arrow[d, "\simeq"]& \Ext_A^1(\Cs_i\otimes_R U_i, N)\arrow[r]\arrow[d]& 0\\
				\Hom_B(G(T_i\otimes_R U_i), GN)\arrow[r]& \Hom_B(GY_i\otimes_R U_i, GN)\arrow[r]&\Ext_B^1(G\Cs_i\otimes_R U_i, GN)\arrow[r] & 0
			\end{tikzcd}.
		\end{center} It follows by diagram chasing that $\Ext_B^1(G(\Cs_i\otimes_R U_i), GN)\simeq \Ext_A^1(\Cs_i\otimes_R U_i, N)$.
		
		Now consider $X\in \mathcal{F}_B(\Stsim_B)$. Then, there is a filtration \begin{align}
			0\subset X_1\subset \cdots \subset X_n=X, \quad X_i/X_{i-1}\simeq \St_B(i)\otimes_R U_i, \ U_i\in R\proj.
		\end{align}We claim that there exists $N\in \mathcal{F}_A(\Cssim)$ such that $GN=X$. We will prove it by induction on the size of the filtration of $X$. It is clear that
		$
		X_1=\St_B(1)\otimes_R U_1\simeq G\Cs_1\otimes_R U_1\simeq G(\Cs_1\otimes_R U_1).
		$  Assume that the result holds for $X_{i-1}$. Consider the exact sequence 
		\begin{align}
			0\rightarrow X_{i-1}\rightarrow X_{i}\rightarrow \St_B(i)\otimes_R U_i\rightarrow 0.\label{eq176}
		\end{align}Here, $\St_B(i)\otimes_R U_i\simeq G(\Cs_i\otimes_R U_i)$. By induction, $X_{i-1}\simeq GN_{i-1}$ for some $N_{i-1}\in \mathcal{F}(\Cssim)$. So, the exact sequence in (\ref{eq176}) belongs to $\Ext_B^1(G(\Cs_i\otimes_R U_i), GN_{i-1})$. Hence, there exists an exact sequence 
		\begin{align}
			0\rightarrow N_{i-1}\rightarrow N_i\rightarrow \Cs_i\otimes_R U_i\rightarrow 0
		\end{align} and its image by $G$ is isomorphic to (\ref{eq176}). In particular, $GN_i\simeq X_i$. It follows that  $G\colon \mathcal{F}(\Cssim)\rightarrow \mathcal{F}(\Stsim_B)$ is dense and we obtain (a).
		
		By Corollary \ref{homofstandardcostandardisproj}, $\St_B(\l)=\Hom_A(T, \Cs(\l))\in R\proj$. Further, $GT\simeq B$ is a progenerator of $B\m$.
		
		Assume that $\Hom_B(\St_B(\l'), \St_B(\l''))\neq0$. Then, $
		0\neq \Hom_B(G\Cs(\l'), G\Cs(\l''))\simeq \Hom_A(\Cs(\l'), \Cs(\l'')).
		$ By Theorem \ref{quasihereditaryintermsofcostandards}, $\l'\geq \l''$. Thus, $\l'\leq_R \l''$.

		Since $Y(\l)\in \mathcal{F}(\Cssim_{\mu<\l})$, it follows that $GY(\l)\in \mathcal{F}(\Stsim_{B_{\mu<\l}})=\mathcal{F}(\Stsim_{B_{\mu>_B\l}})$.
		As $T(\l)\in \add T$, it follows that $GT(\l)\in B\proj$. So, the exact sequence $
		0 \rightarrow GY(\l)\rightarrow GT(\l)\rightarrow \St_B(\l)\rightarrow 0
		$ satisfies $(iv)$ of Definition \ref{qhdef}. Since $G$ is full and faithful on $\mathcal{F}(\Cssim)$, the following holds \begin{align}
			\End_B(\St_B(\l))\simeq \End_B(G\Cs(\l))\simeq \End_A(\Cs(\l))\simeq R.
		\end{align} Thus, (b) holds.
	\end{proof}

	The \textbf{Ringel dual of a split quasi-hereditary algebra} $(A, \{\Delta(\lambda)_{\lambda\in \Lambda}\})$ (or just of $A$ when the there is no ambiguity on the underlying quasi-hereditary structure), is, up to Morita equivalence, the endomorphism algebra $\End_A(T)^{op}$ of a characteristic tilting module $T$ of $A$. We will write $R(A)$ to denote a Ringel dual of $A$.  As in the classical case \citep[Theorem 7]{MR1128706} computing the Ringel dual of a Ringel dual yields back the original split quasi-hereditary structure.

	\begin{Prop}\label{ringeldualsplitqhmorita}
		Let $(A, \{\Delta(\lambda)_{\lambda\in \Lambda}\})$ be a split quasi-hereditary algebra. Let $R(A)$ be a Ringel dual of $A$. Then, $R(R(A))$ is Morita equivalent to $A$ as split quasi-hereditary algebra.
	\end{Prop}
	\begin{proof}
		Define $I=\sumSt I(\l)$, where $I(\l)$ is an $(A, R)$-injective module in the conditions of Theorem \ref{quasihereditaryintermsofcostandards}. In particular, each $I(\l)\in \mathcal{F}(\Cssim)$. We will denote by $G$ and $B$ the functor $\Hom_A(T, -)$ and the Ringel dual $R(A)$, respectively. By Lemma \ref{functorringeldualpart1}(b), $GI\in \mathcal{F}(\Stsim_B)$. By Lemma \ref{functorringeldualpart1}, $\Ext_B^1(G\Cs(\l), GN)\simeq \Ext_A^1(\Cs(\l), N)$ for every $N\in \mathcal{F}(\Cssim_B)$, for every $\l\in \L$. In particular, for $N=I$, and for every $\l\in \L$,
		\begin{align}
			\Ext_B^1(\St_B(\l), GI)\simeq \Ext_A^1(\Cs(\l), I)\simeq \Ext_{(A, R)}^1(\Cs(\l), I)=0.
		\end{align}
		By Theorem \ref{filtrationsintermsofext}, $GI\in \mathcal{F}(\Cssim_B)$. Hence, $GI$ is a partial tilting module. Applying $G$ to the exact sequence
		\begin{align}
			0\rightarrow \Cs(\l)\rightarrow I(\l)\rightarrow C(\l)\rightarrow 0
		\end{align}we obtain the exact sequence
		$
		0\rightarrow \St_B(\l)\rightarrow GI(\l)\rightarrow GC(\l)\rightarrow 0
		$ with $GC(\l)\in \mathcal{F}(\Cssim_{B_{\mu>\l}})=\mathcal{F}(\Cssim_{B_{\mu<_B \l}})$. Therefore, $GI$ is a characteristic tilting module. Furthermore, since $G$ is full and faithful on $\mathcal{F}(\Cssim_A)$, we can write $\End_B(GI)^{op}\simeq \End_A(I)^{op}\simeq \End_{A}(DI)$ which is Morita equivalent to $\End_{A}(A)\simeq A$ because $\add DI=\add A_A$.  
		
		Note that $R(R(A))$ is isomorphic to $\End_{A}(DI)\simeq \End_A(P_{op})$ as $R$-algebras, where $P_{op}$ is the progenerator $\bigoplus_{\l\in \L} P_{op}(\l)$ making $(A^{op}, D\Cs(\l))$ a split quasi-hereditary algebra. So, the equivalence of categories is given by the functor $\Hom_A(\Hom_A(P_{op}, A), -)\colon A\m\rightarrow R(R(A))\m$. Denote this functor by $H$. 
		It is enough to prove that $\Hom_A(\Hom_A(P_{op}, A), \St(\l))\simeq \St_{R(R(A))}(\l)$ for every $\l\in \L$.
		
		Observe that, if $\l\in \L$ is maximal, then $D\Hom_A(\St(\l), A)\simeq I(\l)$ (see Remark \ref{nakayamaonstandardmaximal}). Thus,  \begin{align}
			H\St(\l)&\simeq \Hom_A(\Hom_A(\St(\l), A), P_{op})\simeq \Hom_A(I, D\Hom_A(\St(\l), A))\simeq \Hom_A(I, I(\l))\\&\simeq \Hom_{R(A)}(GI, GI(\l))\simeq \St_{R(R(A))}(\l).
		\end{align} Assume that $|\L|>1$. Let $J$ be the split heredity ideal associated with $\St(\l)$. Denote by $H_J$ the functor $$\Hom_A(\Hom_A(\bigoplus_{\mu\in \L\backslash \{ \l\}} P(\mu)/JP(\mu), A/J ), -)\colon A/J\m\rightarrow R(R(A))\m.$$ By induction, $H_J\St(\mu)\simeq \St_{R(R(A/J))}(\mu)=\St_{R(R(A))}(\mu)$ for every $\mu\neq \l$, $\mu\in \L$. Hence, it is enough to check that $H_JX\simeq HX$ for all $X\in A/J\m$.  Since $J=J^2$ it follows $\Hom_A(P(\mu)/JP(\mu), A/J)\simeq \Hom_A(P(\mu), A/J)$ for every $\mu\in \L$ and $\Hom_A(P(\l), A/J)=0$ (see for example \citep[Proposition 4.7]{Rouquier2008}). Therefore, $$\Hom_A(\bigoplus_{\mu\in \L\backslash \{ \l\}} P(\mu)/JP(\mu), A/J )\simeq \Hom_A(P_{op}, A/J).$$ Moreover, $\Hom_A(\Hom_A(P_{op}, J), X)=0$ for all $X\in A/J\m$.
		Thus, $HX\simeq H_JX$ for every $X\in A/J\m$. Hence, $H$ sends $\St(\mu)$ to $\St_{R(R(A))}(\mu)$ for all $\mu\in \L$. 
	\end{proof}

	\begin{Cor}Let $(A\m, \{\Delta(\lambda)_{\lambda\in \Lambda}\})$ and $(B\m, \{\Omega(\chi)_{\chi\in X}\})$ be two split highest weight categories. 
		$B$ is a Ringel dual of $A$ if and only if there is an exact equivalence between the categories $\mathcal{F}(\Stsim)$ and $\mathcal{F}(\tilde{\mho}
		_{B}),$\label{standardstocostandardsringel}
		where $\mho$ denotes the set of costandard modules of $B$.
	\end{Cor}
	\begin{proof}Let $B=R(A)$ be a Ringel dual of $A$. 
		By Lemma \ref{functorringeldualpart1}, there is an exact equivalence \linebreak\mbox{$\mathcal{F}(\Cssim_{R(A)})\simeq \mathcal{F}(\Stsim_{R(R(A))})$.} By  Proposition \ref{ringeldualsplitqhmorita}, there is an exact equivalence $\mathcal{F}(\Stsim_{R(R(A))})\simeq \mathcal{F}(\Stsim_A)$.
		
		Conversely, assume that there is an exact equivalence between the categories $\mathcal{F}(\Stsim_A)$ and $\mathcal{F}(\tilde{\mho}_{B}).$ By Lemma \ref{functorringeldualpart1}, there is an exact equivalence 
		$
		\mathcal{F}(\Stsim_A)\simeq \mathcal{F}(\tilde{\mho}_{B})\simeq \mathcal{F}(\tilde{\Omega}_{R(B)}).
		$ By Theorem \ref{standardsgivesthewholealgebra}, $R(B)\m$ and $A\m$ are equivalent as split highest weight categories. By Proposition \ref{ringeldualsplitqhmorita}, we conclude that $B\m$ and $R(A)\m$ are equivalent as split highest weight categories.
	\end{proof}

	\subsection{Ringel self-duality} We say that a split quasi-hereditary algebra $(A, \{\Delta(\lambda)_{\lambda\in \Lambda}\})$ over a commutative Noetherian ring $R$ is \textbf{Ringel self-dual} if $A\m$ and $R(A)\m$ are equivalent as split highest weight categories, where $R(A)$ is a Ringel dual of $A$. In view of Corollary \ref{standardstocostandardsringel}, $A$ is Ringel self-dual if and only if there exists an exact equivalence between $\mathcal{F}(\Stsim)$ and $\mathcal{F}(\Cssim)$.

	\begin{Lemma}\label{ringeldualintermsofresiduefield}
		Let $R$ be a local commutative Noetherian ring with maximal ideal $\mi$. Let $(A, \{\Delta(\lambda)_{\lambda\in \Lambda}\})$ be a split quasi-hereditary $R$-algebra. Then, $A$ is Morita equivalent to its Ringel dual as split quasi-hereditary algebra if and only if $A(\mi)$ is Morita equivalent to its Ringel dual as split quasi-hereditary algebra.
	\end{Lemma} 
	\begin{proof}
		Assume that $A$ is Morita equivalent to its Ringel dual as split quasi-hereditary algebra. That is, there exists a progenerator $P$ of $A\m$ so that the Ringel dual of $A$ is the endomorphism algebra $\End_A(P)^{op}$ and $\Hom_A(P, -)\colon A\m\rightarrow R(A)\m$ is an equivalence of categories that sends $\St_A(\l)$ to $\St_{R(A)}(\phi(\l))$ for some bijection $\phi\colon \L\rightarrow \L^{op}$. Hence,
		$P(\mi)$ is a progenerator of $A(\mi)$ and \begin{align}
			\End_{A(\mi)}(P(\mi))^{op}\simeq \End_A(P)^{op}\otimes_R R(\mi)\simeq \End_A(T)^{op}\otimes_R R(\mi)\simeq \End_{A(\mi)}(T(\mi))^{op}.
		\end{align} Moreover, $$\Hom_{A(m)}(P(\mi), \St(\l)(\mi)) \simeq \St_{R(A)}(\phi(\l))\otimes_R U_\l (\mi)\simeq \St_{R(A)}(\phi(\l))(\mi).$$  Hence, $A(\mi)$ is Ringel self-dual. 
		
		Conversely, assume that $A(\mi)$ is Morita equivalent to its Ringel dual as split quasi-hereditary algebras. Since $A$ is semi-perfect we can assume that the projective modules $P(\l)$ are the projective covers of $\St(\l)$. Hence, if $P_{(\mi)}$ is the progenerator giving the Morita equivalence between $A$ and its Ringel dual, we can choose $P\in A\m$ so that $P(\mi)\simeq P_{(\mi)}$. In particular, $P$ is a progenerator of $A$ and for every $\l\in \L$,\begin{align}
			\Hom_A(P, \St(\l))(\mi)\simeq \Hom_{A(\mi)}(P_{(\mi)}, \St(\l)(\mi))\simeq \St_{R(A)}(\phi(\l))(\mi)
		\end{align}for some bijection $\phi\colon \L\rightarrow \L^{op}$. 
		By Propositions \ref{standardscotiltingsreductiontofields} and \ref{standardsrreductionfields}, $\Hom_A(P, \St(\l))\simeq \St_{R(A)}(\phi(\l))$ since the Picard group of $R$ is trivial. 
	\end{proof}
	
	\subsection{Endomorphism algebras of partial tilting modules}  We say that an algebra $A$ over a commutative Noetherian ring $R$ has a \textbf{duality} $\omega$  if $\omega\colon A\rightarrow A $ is an anti-isomorphism of algebras that satisfies $\omega^2=\id_A$ and it fixes a set of orthogonal idempotents $\{e_1, \ldots, e_t \}$ with the following property: for each maximal ideal $\mi$ of $R$ $\{e_1^\mi,\ldots, e_t^\mi  \}$ is a complete set of primitive orthogonal idempotents of $A/\mi A$, where $e_i^\mi$ denotes the image of $e_i$ in $A/\mi A$. 
	
	We say that $(A, \{e_1, \ldots, e_t\})$ is a \textbf{split quasi-hereditary algebra with a duality} $\iota$ if $\iota$ is a duality of $A$ (with respect to the set of orthogonal idempotents $\{e_1, \ldots, e_t\})$ and $A$ is split quasi-hereditary with split heredity chain $
	0\subset Ae_tA\subset \cdots \subset A(e_1+\cdots+ e_t)A=A. 
	$  In such a case, we also say $(A, \{\Delta(\lambda)_{\lambda\in \Lambda}\})$ is a split quasi-hereditary algebra with a duality $\iota$ if $(A, \{e_1, \ldots, e_t\})$ is a split quasi-hereditary algebra with a duality $\iota$ with standard modules $\St(\l)$, $\l\in \L.$
	
	We say that $(A, \{\Delta(\lambda)_{\lambda\in \Lambda}\})$ is a \textbf{split quasi-hereditary cellular algebra } if $(A, \{\Delta(\lambda)_{\lambda\in \Lambda}\})$ is a split quasi-hereditary algebra and a split heredity chain of $(A, \{\Delta(\lambda)_{\lambda\in \Lambda}\})$  is also a cell chain of $A$ with respect to some involution.
	
	The following result indicates that endomorphism algebras of partial tilting modules (which become indecomposable under extension of scalars to fields) over a split quasi-hereditary algebra with a duality are in some sense locally cellular. The classical case can be found in \cite{MR3708257} and \citep[Theorem 1.1]{MR3787549}. For a background to cellular algebras, we refer to \cite{zbMATH00871761, zbMATH01218863, zbMATH01463504}. 
	
	\begin{Theorem}\label{endomorphismoftiltingiscellular}
		Let $R$ be a Noetherian commutative local ring and $A$ a projective Noetherian $R$-algebra. Assume that there exists a set of orthogonal idempotents $\{e_1, \ldots, e_t\}$ so that $(A, \{e_1, \ldots, e_t\})$ is a split quasi-hereditary algebra with a duality $\iota$. Assume that $T$ is a characteristic tilting module of $A$ so that each summand $T(i)(\mi)$ is indecomposable over $A(\mi)$ for every $\mi\in \MaxSpec R$. Let $M\subset T$ be the direct sum of a finite number of summands $T(i)$, $i=1, \ldots, t$.
		Then, $\End_A(M)^{op}$ is a cellular algebra.
	\end{Theorem}
	\begin{proof}
		The duality $\iota$ induces a functor ${}^\iota(-) \colon A\m \rightarrow A^{op}\m$. In particular, ${}^\iota P(i)={}^\iota(Ae_i)=e_iA$ (see for example \cite{zbMATH05871076}) and so it follows that $\iota$ preserves the split quasi-hereditary structure. Consider the contravariant functor ${}^\natural(-)\colon A\m \rightarrow A\m$ given by $D\circ {}^\iota (-)$. So, ${}^\natural (-) $ is a duality functor that interchanges $\St(i)$ with $\Cs(i)$ and as in Lemma 3.2 of \cite{zbMATH05871076} ${}^\natural T(i)\simeq T(i)$. Let $s\colon T\rightarrow {}^\natural T$ be an isomorphism of $A$-modules. Denote by $\alpha\colon \End_A(T)\rightarrow \End_A({}^\natural T)$ the isomorphism of $R$-algebras, given by $\alpha(f)=s\circ f\circ s^{-1}$, $f\in \End_A(T)$ and denote by $\beta\colon \End_A(T)\rightarrow \End_A({}^\natural T)$ the anti-isomorphism of $R$-algebras, given by $\beta(f)(h)(t)=h(f(t))$, $h\in DT$, $t\in T$. Put $\tau=\beta^{-1}\circ \alpha$. By Proposition 2.4 of \cite{zbMATH05871076}, $\tau$ is a duality of the Ringel dual $R(A):=\End_A(T)^{op}$. That is, $\tau$ fixes all maps $T\twoheadrightarrow T(i)\hookrightarrow T$ for every $i$, and $\tau^2=\id_{R(A)}$.
		In particular, $\tau$ fixes the idempotent $f$ of $R(A)$ such that $\Hom_A(T, M)\simeq R(A)f$. Observe that $R(A)$ is split quasi-hereditary with standard modules $\Hom_A(T, \Cs(i))$ with the reversed order on $\{1, \ldots, t\}$. Thus, if we denote by $f_i$ the idempotents $T\twoheadrightarrow T(i)\hookrightarrow T$, $R(A)$ has the split heredity chain
		\begin{align}
			0\subset R(A) f_1 R(A)\subset \cdots \subset R(A) (f_1+\ldots+ f_t)R(A)=R(A).
		\end{align} 
		By \citep[Proposition 4.0.1]{cruz2021cellular}, $R(A)$ is a cellular algebra. By \citep[Proposition 2.2.11]{cruz2021cellular}, $\End_A(M)^{op}\simeq \End_{R(A)}(\Hom_A(T, M))^{op}$ is a cellular algebra.
	\end{proof}  
	
	Such condition on partial tilting modules holds for example if the ground ring is local regular with Krull dimension one.
	\begin{Prop}\label{locallytiltings}
		Let $R$ be a Noetherian commutative local regular ring with Krull dimension one and with maximal ideal $\mi$. Let $(A, \{\Delta(\lambda)_{\lambda\in \Lambda}\})$ be a split quasi-hereditary $R$-algebra. For each $\l\in \L$ there are partial tilting modules $T(\l)$ so that $T(\l)(\mi)$ is indecomposable over $A(\mi)$.
	\end{Prop}
	\begin{proof}
		Denote by $T_{(\mi)}(\l)$ the indecomposable partial tilting module that fits into the exact sequence $0\rightarrow Z(\l)\rightarrow T_{(\mi)}(\l) \xrightarrow{\pi_\l^\mi} \Cs(\l)(\mi)\rightarrow 0$ and by $\pi_\l$ the surjective map $T(\l)\rightarrow \Cs(l)$ given in (\ref{eq121b}). Since both maps $\pi_\l^\mi$ and $\pi_\l(\mi)$ are surjective right $\mathcal{F}(\St(\mi))$-approximations of $\Cs(\l)(\mi)$ we obtain that there are maps $g\in \Hom_{A(\mi)}(T_{(\mi)}(\l), T(\l)(\mi) )$, $f\in \Hom_{A(\mi)}(T(\l)(\mi), T_{(\mi)}(\l) )$ satisfying $\pi_\l^\mi=\pi_\l\circ g$ and $\pi_\l=\pi_\l^\mi\circ g$. In particular, $\pi_\l^\mi=\pi_\l^\mi\circ f\circ g$. Since $\pi_\l^\mi$ is a minimal right approximation it follows that $f\circ g$ is an isomorphism. Therefore, $T_{(\mi)}(\l)$ is a summand of $T(\l)(\mi)$. Using Nakayama's Lemma, consider $X$ to be an $A$-submodule of $T(\l)$ satisfying $X(\mi)\simeq T_{(\mi)}(\l)$. Since $\gldim R\leq 1$ it follows that $X\in A\m\cap R\proj$. By Proposition \ref{standardscotiltingsreductiontofields}, $X$ is a partial tilting module of $A$. So, we can replace $T(\l)$ by $X$ because $R(\mi)\otimes_R -$ preserves filtrations by standard modules and so $X$ must also fit into exact sequences of the form (\ref{eq121}) and (\ref{eq121b}).
	\end{proof}

	In particular, a Ringel dual of a split quasi-hereditary algebra with a duality $\iota$ over a local commutative Noetherian regular ring with Krull dimension at most one is Morita equivalent as split quasi-hereditary algebra to a split quasi-hereditary cellular algebra.

\section*{Acknowledgments}
Most of the results of this paper are contained in the author's PhD thesis \citep{thesis}, financially supported by \textit{Studienstiftung des Deutschen Volkes}. The author would like to thank Steffen Koenig for all the conversations on these topics, his comments and suggestions towards improving this manuscript.

\bibliographystyle{alphaurl}
\bibliography{bibarticle}

\Address
\end{document}